\documentclass[a4paper, 11pt, english]{article}
\usepackage[utf8]{inputenc}
\usepackage[T1]{fontenc}
\usepackage{babel}
\usepackage{graphicx}
\usepackage{stmaryrd}
\usepackage[a4paper]{geometry}
\usepackage{amsmath,amsfonts,amssymb,amsthm,epsfig,epstopdf,url,array}
\usepackage{rotating}
\usepackage[colorlinks=true,citecolor=red,linkcolor=blue,pdfpagetransition=Blinds]{hyperref}
\usepackage{cleveref}
\usepackage{nameref}
\usepackage{enumitem}
\usepackage{comment}
\Crefname{paragraph}{Section}{Sections}
\usepackage{fancyhdr}
\usepackage{tikz}
\usetikzlibrary{calc, decorations.pathreplacing, patterns, arrows.meta}

\usepackage{fullpage}

\usepackage{mdframed}

\newcommand{\ensemblenombre}[1]{\mathbb{#1}}

\newcommand{\R}{} 
\renewcommand{\R}{\ensemblenombre{R}}
\newcommand{\C}{\ensemblenombre{C}}

\renewcommand{\leq}{\leqslant}
\renewcommand{\geq}{\geqslant}

\newcommand{\dive}[1]{\mathrm{div}}

\newcommand{\eps}{\varepsilon}

\theoremstyle{plain} 
\newtheorem{pr}{Proposition}[section] 
\newtheorem{tm}[pr]{Theorem}
\newtheorem{lm}[pr]{Lemma}
\newtheorem{cor}[pr]{Corollary}
\theoremstyle{definition}

\newtheorem{rmk}[pr]{Remark}

\newtheorem{conj}[pr]{Conjecture}

\numberwithin{equation}{section}

\makeatletter
\let\original@addcontentsline\addcontentsline
\newcommand{\dummy@addcontentsline}[3]{}
\newcommand{\DeactivateToc}{\let\addcontentsline\dummy@addcontentsline}
\newcommand{\ActivateToc}{\let\addcontentsline\original@addcontentsline}
\makeatother

\begin{document}

\author{Kévin Le Balc'h}

\title{Almost sharp local Bernstein estimates for Laplace eigenfunctions on compact Riemannian manifolds}

\maketitle

\begin{abstract}
We study local growth properties of Laplace eigenfunctions on compact Riemannian manifolds. Following the paradigm introduced by Donnelly and Fefferman in the late 1980s, an eigenfunction is expected to behave locally like a polynomial of degree comparable to the square root of the eigenvalue. In this direction we establish almost sharp local $L^{p}$--Bernstein inequalities, $p\in[1,\infty]$, conjectured by Donnelly--Fefferman in 1990. We also derive analogous estimates for $A$-harmonic functions, with the square root of the eigenvalue replaced by the doubling index.

Our argument refines the original Donnelly--Fefferman method based on $L^{2}$--Carleman estimates. At the $L^{2}$--level, we first prove a uniform bound for the doubling index on annuli of width comparable to the wavelength. This implies, with an arbitrarily small polynomial loss, the corresponding property at the $L^{p}$--level for all $p\in[1,\infty]$. The latter step relies on a bootstrap scheme combining elliptic regularity with a patching of local Carleman estimates on small balls.
\end{abstract}

\small

\tableofcontents

\normalsize

\section{Introduction}

Let $M$ be a $C^{\infty}$-smooth, compact, connected, Riemannian manifold of dimension $d \geq 2$, without boundary, equipped with a Riemannian metric $g$. In this article we are interested in growth properties of Laplace eigenfunctions $\varphi_{\lambda} \in C^{\infty}(M)$ associated to the eigenvalue $\lambda \geq 0$, 
\begin{equation}
\label{eq:Laplaceeingenfunction}
- \Delta_g \varphi_{\lambda} = \lambda \varphi_{\lambda}\ \text{in}\ M,
\end{equation}
where $\Delta_g= \mathrm{div}_g \circ \nabla_g$ is the Laplace-Beltrami operator. 
One may distinguish between global growth and local growth. 

\medskip

\noindent \textbf{Global Bernstein estimate.} The famous classical Bernstein's estimate on trigonometric polynomials is typically a global growth estimate on linear combination of Laplace eigenfunctions on the one-dimensional torus. Let $n \geq 0$ and $T = \sum_{k=-n}^{n} a_k e^{ikx}$ for $x \in (0,2 \pi)$, then
\begin{equation}
\label{eq:classicalBernstein}
\sup_{x \in (0, 2 \pi)} |T'(x)| \leq n \sup_{x \in (0, 2 \pi)} |T(x)|.
\end{equation}
For a survey around \eqref{eq:classicalBernstein}, one can read \cite{QZ19}. Global $L^{\infty}$-Bernstein estimates also hold for a single Laplace eigenfunction on $M$, i.e. there exist $C>0$ depending only on $M$ such that for every Laplace eigenfunction $\varphi_{\lambda}$, 
\begin{equation}
\label{eq:globalBernstein}
\sup_{M} |\nabla \varphi_{\lambda}| \leq C \sqrt{\lambda} \sup_{M} | \varphi_{\lambda}|.
\end{equation}
This estimate \eqref{eq:globalBernstein} is actually a consequence of standard elliptic estimates applied to the harmonic lifted function $u(x,t) = \varphi_{\lambda}(x) e^{\sqrt{\lambda}t}$, see \cite[Corollary 3.3]{OCP13} for a proof. Note that one can actually extend global Bernstein estimates for a single Laplace eigenfunction that is \eqref{eq:globalBernstein} to a linear combination of Laplace eigenfunctions, i.e. there exist $C>0$ depending only on $M$ such that for $\Phi_{\Lambda} = \sum_{\lambda_k \leq \Lambda} a_k \varphi_{\lambda_k}$, then 
\begin{equation}
\label{eq:globalBernsteinSum}
\sup_{M} |\nabla \Phi_{\Lambda}| \leq C \sqrt{\Lambda} \sup_{M} | \Phi_{\Lambda}|,
\end{equation}
see for instance \cite[Theorem 2.1]{FM10} and \cite[Theorem 1.2]{IO22}. In this latter case, the proofs are considerably more involved.

\subsection{On local growth and local Bernstein estimates of Laplace eigenfunctions}
 
Concerning local growth, from the breakthrough work of Donnelly, Fefferman \cite{DF88}, we know that $\varphi_{\lambda}$ also shares local growth properties with a polynomial function of degree proportional to $\sqrt{\lambda}$. One of the most celebrated result is the following bound on the doubling index of Laplace eigenfunctions, see \cite[Theorem 4.2]{DF88}.\medskip 

\noindent \textbf{Bound on the doubling index.} There exist $r_0, C >0$ depending only on $M$, such that for every Laplace eigenfunction $\varphi_{\lambda} \in C^{\infty}(M)$, i.e. satisfying \eqref{eq:Laplaceeingenfunction}, for every $x \in M$, $r \in (0,r_0)$,
\begin{equation}
\label{eq:doublingindexLaplaceeigen}
\sup_{B_g\left(x,2 r\right)} |\varphi_{\lambda}| \leq e^{C \sqrt{\lambda}} \sup_{B_g\left(x,r\right)} |\varphi_{\lambda}|.
\end{equation}
\begin{rmk}
Note that \eqref{eq:doublingindexLaplaceeigen} is in perfect agreement to the previous heuristics because $$\sup_{t \in (-2r,2r)} t^{\sqrt{\lambda}} = 2^{\sqrt{\lambda}} \sup_{t \in (-r,r)} t^{\sqrt{\lambda}}.$$
\end{rmk}

For $f \in C^{\infty}(M)$, $x \in M$, $r >0$, the number
\begin{equation*}
N_f(B_g(x,r)) := \log\left(\frac{\sup_{B_g\left(x,2 r\right)} |f|}{\sup_{B_g\left(x,r\right)} |f|}\right),
\end{equation*}
is usually called the doubling index of $f$ in the ball $B_g(x,r)$. Note that for $x \in M$,
\begin{equation*}
\lim_{r \to 0} N_f(B_g(x,r)) = \text{vanishing order of }f\ \text{at}\ x,
\end{equation*}
where the vanishing order of $f$ at $x$ is the smallest integer $k$ such that the derivatives of $f$ of order smaller than $k$ vanish while there is some non-zero derivative of order $k$. As a consequence, the doubling index estimate \eqref{eq:doublingindexLaplaceeigen} tells us that the vanishing order of $\varphi_{\lambda}$ is bounded by $C \sqrt{\lambda}$. This last result is sharp if we do not make extra assumptions on the Riemannian manifold because the vanishing order of spherical harmonics is comparable to $C \sqrt{\lambda}$.\medskip

In \cite{DF90}, the authors pursue the analogy between Laplace eigenfunctions and polynomial functions. They obtain the following local $L^2$ estimates, see \cite[Theorem 1]{DF90}.

\medskip

\noindent \textbf{Local Bernstein estimates.}  There exist $r_0, C >0$ depending only on $M$, such that for every Laplace eigenfunction $\varphi_{\lambda} \in C^{\infty}(M)$, i.e. satisfying \eqref{eq:Laplaceeingenfunction}, for every $x \in M$, $r \in (0,r_0)$,
\begin{align}
\label{eq:L2BernsteinfunctionAnnulus}
 \|\varphi_{\lambda}(y)\|_{L^2\left(r < d_g(y,x) < r\left(1+\frac{2}{\sqrt{\lambda}}\right)\right)} &\leq C  \|\varphi_{\lambda}(y)\|_{L^2\left(r < d_g(y,x)  < r\left(1+\frac{1}{\sqrt{\lambda}}\right)\right)},\\
\label{eq:L2Bernsteinfunction}
 \|\varphi_{\lambda}\|_{L^2\left(B_g\left(x,r\left(1+\frac{1}{\sqrt{\lambda}}\right)\right)\right)} &\leq C  \|\varphi_{\lambda}\|_{L^2\left(B_g\left(x,r\right)\right)},\\
\label{eq:L2Bernsteingradient}
\|\nabla \varphi_{\lambda}\|_{L^2\left(B_g\left(x,r\right)\right)} &\leq C \frac{\sqrt{\lambda}}{r} \|\varphi_{\lambda}\|_{L^2\left(B_g\left(x,r\right)\right)}.
\end{align}

\begin{rmk}
The inequality \eqref{eq:L2Bernsteingradient} is referred to as a local \(L^2\)-Bernstein estimate, as it shares structural similarities with the classical global \(L^{\infty}\)-Bernstein estimate \eqref{eq:globalBernstein}. It follows directly from \eqref{eq:L2Bernsteinfunction}, which itself is a consequence of \eqref{eq:L2BernsteinfunctionAnnulus}. Furthermore, the classical bound on the doubling index of Laplace eigenfunctions, recalled in \eqref{eq:doublingindexLaplaceeigen}, can also be derived from \eqref{eq:L2Bernsteinfunction}. Hence, among these, \eqref{eq:L2BernsteinfunctionAnnulus} is the strongest estimate. Roughly, the estimate \eqref{eq:L2BernsteinfunctionAnnulus} shows that the doubling index at the \(L^2\)-level for Laplace eigenfunctions, when considered on annuli of width comparable to the wavelength \(\lambda^{-1/2}\), remains \emph{uniformly bounded}. This stands in sharp contrast with the situation on balls, where the bound \eqref{eq:doublingindexLaplaceeigen} is known to be sharp even for \(r = c\,\lambda^{-1/2}\), as can be seen from the case of spherical harmonics on the two-dimensional sphere \(\mathcal{S}^2\). As a consequence, we have that all the estimates \eqref{eq:L2BernsteinfunctionAnnulus}, \eqref{eq:L2Bernsteinfunction} and \eqref{eq:L2Bernsteingradient} are sharp. In \cite[Theorem~1]{DF90}, only the last two inequalities, \eqref{eq:L2Bernsteinfunction} and \eqref{eq:L2Bernsteingradient}, are explicitly stated. The stronger estimate \eqref{eq:L2BernsteinfunctionAnnulus} does not appear there but is implicitly contained in the proof. As will become clear later, this hidden estimate plays a crucial role in extending the results to the \(L^p\)-setting for \(p \in [1, \infty]\). For this reason, we shall recall the proof of \eqref{eq:L2BernsteinfunctionAnnulus} in detail.
\end{rmk}

\medskip

\noindent \textbf{On local Bersnstein estimates at $L^{\infty}$-level.} In \cite{DF90}, motivated by \eqref{eq:L2Bernsteinfunction}, the authors formulate the following conjecture. 
\begin{conj}[\cite{DF90}]
\label{conj:df90}
There exist $r_0, C >0$ depending only on $M$, such that for every Laplace eigenfunction $\varphi_{\lambda} \in C^{\infty}(M)$, i.e. satisfying \eqref{eq:Laplaceeingenfunction}, for every $x \in M$, $r \in (0,r_0)$,
\begin{equation}
\label{eq:LinftyBernsteingradientconj}
\sup_{B_g\left(x,r\right)} |\nabla \varphi_{\lambda}| \leq C \frac{\sqrt{\lambda}}{r}  \sup_{B_g\left(x,r\right)} |\varphi_{\lambda}|.
\end{equation}
\end{conj}
\begin{rmk}
\Cref{conj:df90} is again motivated by the heuristics that $\varphi_{\lambda}$ behaves as $t^{\sqrt{\lambda}}$ because
$$\sup_{t \in(-r,r)}\left(\frac{d}{dt} t^{\sqrt{\lambda}}\right) =  \sup_{t \in(-r,r)} \sqrt{\lambda} t^{\sqrt{\lambda}-1} = \frac{\sqrt{\lambda}}{r} \sup_{t \in(-r,r)} t^{\sqrt{\lambda}}.$$
\end{rmk}

We briefly recall the known local Bernstein estimates at the $L^{\infty}$ scale.
\begin{itemize}
\item In \cite{DF90}, starting from the $L^2$ Bernstein estimate \eqref{eq:L2Bernsteinfunction}, the authors established the weak local $L^{\infty}$ gradient bound
\begin{equation}
\label{eq:weakLinftyBernsteingradient}
\sup_{B_g\left(x,r\right)} |\nabla \varphi_{\lambda}| \leq C \frac{\lambda^{\frac{d+2}{4}}}{r} \sup_{B_g\left(x,r\right)} |\varphi_{\lambda}|.
\end{equation}
\item In \cite{Don95}, Dong refines \eqref{eq:weakLinftyBernsteingradient} in dimension $d=2$, using geometric ideas originating from \cite{Don92}. He proved
\begin{equation}
\label{eq:weakLinftyBernsteingradientDong}
\sup_{B_g\left(x,r\right)} |\nabla \varphi_{\lambda}| \leq C \max\left(\frac{\sqrt{\lambda}}{r}, \lambda^{\frac 34}\right) \sup_{B_g\left(x,r\right)} |\varphi_{\lambda}|\qquad (d=2).
\end{equation}
\item In the very recent article \cite{DM23}, Decio and Malinnikova obtained further improvements. They proved in dimension $d=2$
\begin{equation}
\label{eq:weakLinftyBernsteingradientDMsurfaces}
\sup_{B_g\left(x,r\right)} |\nabla \varphi_{\lambda}| \leq C \max\left(\frac{\sqrt{\lambda}}{r}, \sqrt{\lambda} \log(\lambda)\right) \sup_{B_g\left(x,r\right)} |\varphi_{\lambda}|\qquad (d=2),
\end{equation} 
and in all dimensions $d \ge 2$,
\begin{equation}
\label{eq:weakLinftyBernsteingradientDM}
\sup_{B_g\left(x,r\right)} |\nabla \varphi_{\lambda}| \leq C \max\left(\frac{\sqrt{\lambda}}{r} \log^{2}(\lambda), \lambda \log^{2}(\lambda)\right) \sup_{B_g\left(x,r\right)} |\varphi_{\lambda}|\qquad (d \geq 2).
\end{equation} 
\end{itemize}

\begin{rmk}
The estimate \eqref{eq:weakLinftyBernsteingradientDMsurfaces} is a strong refinement of \eqref{eq:weakLinftyBernsteingradientDong} and gives \Cref{conj:df90} for surfaces up to a logarithm loss and \eqref{eq:weakLinftyBernsteingradientDM} gives \Cref{conj:df90} up to a logarithm loss at the wavelength scale, i.e. $r \in \left(0, r_0 \sqrt{\lambda}^{-1}\right)$ while \eqref{eq:weakLinftyBernsteingradientDM} resembles more to the Markov's inequality for polynomials at larger scales in any dimension. Recall that for an algebraic polynomial of degree $n$, the Markov inequality holds
\begin{equation}
\label{eq:markov}
\sup_{x \in (-1,+1)} |P_n'(x)| \leq n^2 \sup_{x \in (-1,+1)} |P_n'(x)|.
\end{equation}
Note that \eqref{eq:markov} is sharp because Chebychev polynomials are extremizers of this inequality.
\end{rmk}

\subsection{Main results}

\noindent \textbf{Almost sharp local $L^p$-estimates for Laplace eigenfunctions.} The first main result of this paper is the establishment of Donnelly, Fefferman's conjecture on $L^{p}$-Bernstein estimates for Laplace eigenfunctions, $p \in [1, \infty]$, up to some arbitrary small-loss.
\begin{tm}
\label{tm:mainresult1}
For every $p \in [1,\infty]$, $\varepsilon>0$, there exist $r_0, C >0$ depending only on $M$, $p$ and $\varepsilon$, such that for every Laplace eigenfunction $\varphi_{\lambda} \in C^{\infty}(M)$, i.e. satisfying \eqref{eq:Laplaceeingenfunction}, for every $x \in M$, $r \in (0,r_0)$,
\begin{align}
\label{eq:LpBernsteinfunctionAnnulus}
 \|\varphi_{\lambda}(y)\|_{L^p\left(r < d_g(y,x)  < r\left(1+\frac{2}{\sqrt{\lambda}}\right)\right)} &\leq C \lambda^{\varepsilon} \|\varphi_{\lambda}(y)\|_{L^p\left(r < d_g(y,x)  < r\left(1+\frac{1}{\sqrt{\lambda}}\right)\right)},\\
\label{eq:LpBernsteinfunction}
 \|\varphi_{\lambda}\|_{L^p\left(B_g\left(x,r\left(1+\frac{1}{\sqrt{\lambda}}\right)\right)\right)} &\leq C \lambda^{\varepsilon}  \|\varphi_{\lambda}\|_{L^p\left(B_g\left(x,r\right)\right)},\\
\label{eq:LpBernsteingradient}
\|\nabla \varphi_{\lambda}\|_{L^p\left(B_g\left(x,r\right)\right)} &\leq C \frac{\lambda^{1/2+\varepsilon}}{r} \|\varphi_{\lambda}\|_{L^p\left(B_g\left(x,r\right)\right)}.
\end{align}
\end{tm}
\begin{rmk}
For $p=\infty$, the estimate \eqref{eq:LpBernsteingradient} reduces to the local $L^{\infty}$ Bernstein estimate \eqref{eq:LinftyBernsteingradientconj} from \Cref{conj:df90}, up to a loss of order $\lambda^{\varepsilon}$ on the right-hand side.
\end{rmk}

\noindent \textit{\textbf{On the sharp version of \Cref{tm:mainresult1}.}
During the preparation of this work, Eugenia Malinnikova kindly informed me, through a personal communication \cite{DMN-Personal}, of an ongoing work of Decio, Malinnikova, and Nazarov establishing, via a different method, the sharp local Bernstein estimates \eqref{eq:LpBernsteinfunction} and \eqref{eq:LpBernsteingradient}, that is, without the $\lambda^{\varepsilon}$-loss on the right-hand side. In particular, their results confirm \Cref{conj:df90}.
}

\medskip

\noindent \textbf{Almost sharp local $L^p$-estimates for harmonic functions.} We actually prove a similar result for $A$-harmonic functions in the Euclidean space where the role of the square root of the eigenvalue is played by the doubling index, that serves as a local degree of the solution. More precisely, we look at $L^{\infty}$-Bernstein estimates for $A$-harmonic functions with a bounded doubling index. The matrix $A=(a^{ij}(x))_{1 \leq i, j \leq d}$ is supposed to be symmetric, uniformly elliptic, with Lipschitz entries
\begin{equation}
\label{eq:LipschitzAB2}
    \Lambda_{1}^{-1} |\xi|^2 \leq \langle A(x) \xi, \xi \rangle \leq  \Lambda_{1} |\xi|^2,\quad |a^{ij}(x) - a^{ij}(y)| \leq \Lambda_2 |x-y|,\qquad x, y \in B_2,\ \xi \in \R^d,
\end{equation}
for some $\Lambda_1, \Lambda_2 >0$. We focus on functions $u$ in $B_2$, verifying
\begin{equation}
\label{eq:Aharmonic}
- \mathrm{div}(A(x) \nabla u) = 0\ \text{in}\ B_2,
\end{equation} 
with bounded doubling index
\begin{equation}
\label{eq:bounddoublingindex}
N_u(B(0,1)) := \log\left(\frac{\sup_{B\left(0,2\right)} |u|}{\sup_{B\left(0,1\right)} |u|}\right) \leq N,\qquad N \geq 2.
\end{equation}

In the spirit of local $L^2$-Bernstein estimates, we have the following result at $L^2$-level.
\begin{tm}
\label{tm:df90harmonic}
There exist $r_0, C >0$ depending only on $A$ such that for every function $u \in H_{\text{loc}}^1(B_2) \cap L^{\infty}(B_2)$ satisfying \eqref{eq:Aharmonic} with a bounded doubling index $N$ defined in \eqref{eq:bounddoublingindex}, for every $r \in (0,r_0)$,
\begin{align}
\label{eq:L2BernsteinfunctionAnnulusharmonic}
 \|u(x)\|_{L^2\left(r < |x| < r\left(1+\frac{2}{N}\right)\right)} &\leq C  \|u(x)\|_{L^2\left(r < |x| < r\left(1+\frac{1}{N}\right)\right)},\\
\label{eq:L2Bernsteinfunctionharmonic}
 \|u\|_{L^2\left(B\left(0,r\left(1+\frac{1}{N}\right)\right)\right)} &\leq C  \|u\|_{L^2\left(B\left(0,r\right)\right)},\\
\label{eq:L2Bernsteingradientharmonic}
\|\nabla u\|_{L^2\left(B\left(0,r\right)\right)} &\leq C \frac{N}{r} \|u\|_{L^2\left(B\left(0,r\right)\right)}.
\end{align}
\end{tm}


Based on \Cref{tm:df90harmonic}, the second main result of this paper is the following one.
\begin{tm}
\label{tm:mainresult2}
For every $p \in [1,\infty]$, $\varepsilon>0$, there exist $r_0, C >0$ depending only on $A$, $p$ and $\varepsilon$ such that for every function $u \in H_{\text{loc}}^1(B_2) \cap L^{\infty}(B_2)$ satisfying \eqref{eq:Aharmonic} with a bounded doubling index $N$ defined in \eqref{eq:bounddoublingindex}, for every $r \in (0,r_0)$,
\begin{align}
\label{eq:LpBernsteinfunctionAnnulusharmonic}
 \|u(x)\|_{L^p\left(r < |x| < r\left(1+\frac{2}{N}\right)\right)} &\leq C N^{\varepsilon} \|u(x)\|_{L^p\left(r < |x| < r\left(1+\frac{1}{N}\right)\right)},\\
\label{eq:LpBernsteinfunctionharmonic}
 \|u\|_{L^p\left(B\left(0,r\left(1+\frac{1}{N}\right)\right)\right)} &\leq C N^{\varepsilon}  \|u\|_{L^p\left(B\left(0,r\right)\right)},\\
\label{eq:LpBernsteingradientharmonic}
\|\nabla u\|_{L^p\left(B\left(0,r\right)\right)} &\leq C \frac{N^{1+\varepsilon}}{r} \|u\|_{L^p\left(B\left(0,r\right)\right)}.
\end{align}
\end{tm}
\begin{rmk}
\Cref{tm:mainresult2} has to be compared to \cite[Theorem 2]{DM23} where the authors obtain a similar result with stronger regularity assumptions on the matrix $A$ and stronger smallness assumptions on the radius $r$, that has to be small in function of the doubling index. 
\end{rmk}

\begin{rmk}
On the one hand, \Cref{tm:mainresult1} and \Cref{tm:mainresult2} are related by the standard lifting trick that allows to pass from Laplace eigenfunctions to harmonic functions. If $\varphi_{\lambda}$ satisfies \eqref{eq:Laplaceeingenfunction} then the function
\begin{equation}
\label{eq:liftingtrick}
u(x,t) = \varphi_{\lambda}( x) e^{\sqrt{\lambda} t}\qquad (x,t) \in M \times \R,
\end{equation}
is harmonic on the product manifold $M \times \R$ and by using \eqref{eq:doublingindexLaplaceeigen} its doubling index is bounded by $C \sqrt{\lambda}$. This standard trick was first observed by \cite{Lin91} in the study of the nodal volume for Laplace eigenfunctions on compact Riemannian manifolds, it has other applications like for instance the obtaining of the bound on the doubling index of Laplace eigenfunctions in \eqref{eq:doublingindexLaplaceeigen}, see \cite[Proposition 2.4.1]{LM20}. On the other hand, it is worth mentioning that \Cref{tm:mainresult1} is not a direct consequence of \Cref{tm:mainresult2}, as it is the case in \cite{DM23} where the authors deduce $L^{\infty}$-Bernstein estimates for Laplace eigenfunctions from $L^{\infty}$-Bernstein estimates for $A$-harmonic functions because they are working at the wavelength scale, i.e. $r \leq C \sqrt{\lambda}^{-1}$. In our case, the same phenomenon appears,  we can only deduce \Cref{tm:mainresult1} from \Cref{tm:mainresult2} for $r \leq C \sqrt{\lambda}^{-1}$. This is why we will actually prove \Cref{tm:mainresult1} in an independent way by following the same strategy of the proof of \Cref{tm:mainresult2} even if the some new technical difficulties appear. \end{rmk}

\subsection{Strategy of the proof}

The proofs of \Cref{tm:mainresult1} and \Cref{tm:mainresult2} follow similar lines. We first present the argument for \Cref{tm:mainresult2}, and then describe the minor adaptations required for \Cref{tm:mainresult1}. The proof of \eqref{eq:LpBernsteinfunctionAnnulusharmonic} is divided into the following steps.

\medskip
 \textbf{Step 1: Uniform $L^2$-doubling bound on small annuli.}
We begin with the $L^2$ estimate \eqref{eq:L2BernsteinfunctionAnnulusharmonic}. The argument originates in the proof of the $L^2$ Bernstein estimates for Laplace eigenfunctions in \cite{DF90}, which rely on an appropriate $L^2$ Carleman estimate; see in particular \Cref{lm:Carlemanpunctured}.

\medskip
 \textbf{Step 2: Polynomial $L^p$-doubling bound on small annuli.}
Working on annuli of width comparable to $N^{-1}$, the previous step implies that the corresponding $L^p$-doubling index satisfies a polynomial bound. This follows from standard elliptic regularity arguments; see \Cref{pr:doublingpolynomial}. For later purposes, we refine this by proving the same bound on annuli of width $\varepsilon \sim N^{-1}\log N$; see \Cref{pr:doublingpolynomialPratical}.

\medskip
 \textbf{Step 3: $L^2$ Carleman estimates on small balls covering the annulus.}
We cover the annulus by balls of radius $\varepsilon \sim N^{-1}\log N$, apply suitable Carleman estimates in each ball, and sum the contributions. Three types of cut-off terms arise: those supported strictly inside the annulus, those near the outer boundary, and those near the inner boundary. The interior terms are absorbed by taking the Carleman parameter $\alpha$ sufficiently large relative to $\varepsilon^{-1}$, i.e. $\alpha \ge C\varepsilon^{-1}$. The terms near the outer boundary are controlled using the polynomial $L^p$-doubling bound from Step~2, requiring $\alpha \ge C\varepsilon^{-1}\log N \sim N$. Finally, the terms near the inner boundary give rise to an observation term. Comparing both sides, using elliptic regularity and Hölder inequalities on balls of radius~$\varepsilon$, we obtain that the $L^p$-doubling index on annuli of width $\sim N^{-1}$ grows at most like $(\alpha \varepsilon)^K$ for some $K=K(d)>0$, hence at most like $(\log N)^K$. In other words, the polynomial growth from Step~2 is improved to logarithmic growth, yielding \eqref{eq:LpBernsteinfunctionAnnulusharmonic}.

\medskip
\textbf{Extra step: the $L^\infty$ case.}
The case $p=+\infty$ follows by letting $p\to\infty$ and observing that the constants in the previous steps do not depend on $p$.

\begin{rmk}
A more straightforward approach—applying a single Carleman estimate directly on an annulus of width~$\varepsilon$—fails because the volume of such an annulus is of order~$\varepsilon$, and elliptic regularity combined with Hölder inequalities would incur a polynomial loss. By instead working on balls of radius~$\varepsilon$, whose volume is of order~$\varepsilon^d$, these estimates produce only logarithmic losses. As a consequence, one may replace the arbitrary small polynomial loss in the main theorems by a dimensional power of a logarithmic loss, which is strictly smaller. For clarity, we keep the polynomial version in the statements. The strategy may also be iterated to obtain nearly sharp local Bernstein estimates with arbitrarily small loss.
\end{rmk}

\medskip
\noindent \textbf{Extension to Laplace eigenfunctions.}
The proof of \eqref{eq:LpBernsteinfunctionAnnulus} in \Cref{tm:mainresult1} is analogous. The main new difficulty is that the operator becomes $-\Delta_g - \lambda$. We therefore use $L^2$ Carleman estimates adapted to $-\Delta_g - \lambda$, with Carleman parameter satisfying $\alpha \ge C\sqrt{\lambda}$. We also track how elliptic regularity in geodesic balls of radius $\varepsilon \sim \lambda^{-1/2}\log \lambda$ changes for the equation $-\Delta_g \varphi_\lambda - \lambda \varphi_\lambda =0$; this introduces only logarithmic losses in $\lambda$, which are smaller than any power of $\lambda$. Consequently, this do not affect the overall strategy.

\bigskip

\noindent \textbf{Organization of the paper.} In \Cref{sec:proofgrowthestimates}, we present the proofs of our main results, \Cref{tm:mainresult1} and \Cref{tm:mainresult2}. In \Cref{sec:extensions}, we discuss some generalizations/extensions/open problems related to our main results, \Cref{tm:mainresult1} and \Cref{tm:mainresult2}.

\bigskip

\noindent \textbf{Acknowledgements.} I warmly thank Eugenia Malinnikova for fruitful discussions during the preparation of this work.

\section{Proof of the growth estimates}
\label{sec:proofgrowthestimates}

The goal of this part is to prove \Cref{tm:mainresult1} and \Cref{tm:mainresult2}. \medskip

The first four parts are dedicated to the proof of \Cref{tm:mainresult2} while the last part is devoted to the proof of \Cref{tm:mainresult1}. Recall that the proof of the growth estimates for Laplace eigenfunctions on Riemannian manifolds stated in  \Cref{tm:mainresult1} is a small adaptation of the one of the growth estimates for $A$-harmonic functions on the Euclidean space stated in \Cref{tm:mainresult2}, this is why we will only insist on the new difficulties that appear in the sixth part. The first part consists in stating $L^2$-Carleman estimates for the operator $\mathrm{div}(A \nabla \cdot)$, the second part proves vanishing order estimates for $A$-harmonic functions with bounded doubling index, the third part establishes the uniform bound of the doubling index at $L^2$-level on annulus of width comparable to $N^{-1}$, then the polynomial bound at $L^p$-level as a consequence, the fourth part is dedicated to the proof of the almost sharp local $L^p$-estimates on annulus of width comparable to $N^{-1}$, the fifth part consists in proving the expected growth estimates i.e. local Bernstein estimates in $L^p$.\medskip


In the next five parts, the positive constants $C>0$, $c>0$ depend on $A$ and $d$ while in the last part, the positive constants $C>0$, $c>0$ are allowed to depend on $M$, $g$ and $d$. To insist on the dependence of a positive constant $C$ in function of some parameter $s$, we will sometimes use the notation $C=C(s)$. Moreover, the constants can vary from one line to another without explicitly mentioning it.

\subsection{$L^2$-Carleman estimates}

The goal of this part is to state $L^2$-Carleman estimates.

\medskip

To simplify the notations in the next, we set 
\begin{equation*}
\mathrm{div}(A(x) \nabla f) = \Delta_A f.
\end{equation*}
The radial part of $x \in \R^d$ will be denoted by $|x|=r$.

\medskip

First, we have the following standard $L^2$-Carleman estimate, that comes from a direct application of \cite[Theorem 2]{EV03}, stated in the parabolic case.
\begin{lm}
\label{lm:Carleman}
There exists a positive constant $C=C(A)>0$, a radial increasing function $\rho=\rho(r)$ for $0 < r < 2$ satisfying
\begin{equation}
\label{eq:assumptionrho}
C^{-1} \leq \frac{\rho(r)}{r} \leq C,\ C^{-1}  \leq |\partial_r \rho(r) | \leq C\ \qquad \forall r \in (0,2),
\end{equation} 
such that for every $\alpha \geq C$, $f \in C_c^{\infty}(B_{2} \setminus \{0\})$, the following estimate holds
\begin{equation}
\label{eq:Carleman1}
\alpha^3 \int_{B_{2}} \rho^{-1-2 \alpha} |f|^2 dx + \alpha  \int_{B_{2}} \rho^{1- 2 \alpha} |\nabla f|^2 dx  
\leq C  \int_{B_{2}}  \rho^{2-2 \alpha} |\Delta_A f|^2 dx.
\end{equation}
\end{lm}

In the next, the function $\rho$ introduced in \Cref{lm:Carleman} is fixed.

\medskip

The next result tells us how the Carleman estimate \eqref{eq:Carleman1} from \Cref{lm:Carleman} translates when the function vanishes in a small ball centered at $0$.
\begin{lm}
\label{lm:Carlemanpunctured}
There exists a positive constant $C=C(A)>0$ and $c=c(A)>0$ such that for every $r \in (0,c)$, $\varepsilon \in (0,1)$, $\alpha \geq C \varepsilon^{-1}$, for all $f \in C_c^{\infty}(B_{2} \setminus B_r)$, the following estimate holds
\begin{equation}
\label{eq:secondCarleman}
\frac{\alpha^2}{\varepsilon^2 r^2} \int_{B(0,r(1+2\varepsilon))} \rho^{-2 \alpha} |f|^2 dx + \frac{1}{\varepsilon^2} \int_{B(0,r(1+2\varepsilon))} \rho^{-2 \alpha} |\nabla f|^2 dx 
\leq C  \int_{B_{2}} \rho^{2-2 \alpha} |\Delta_A f|^2 dx.
\end{equation}
\end{lm}
\Cref{lm:Carlemanpunctured} is inspired by \cite[Lemma A]{DF90}. Nevertheless, we make one crucial improvement. While \cite[Lemma A]{DF90} only considers the case $\varepsilon = \alpha^{-1}$, our result from \Cref{lm:Carlemanpunctured}, allows us to choose parameters $\varepsilon \in (0,1)$, $\alpha \geq C \varepsilon^{-1}$. 

\medskip The proof of \Cref{lm:Carlemanpunctured} is rather the same as the one of \cite[Lemma A]{DF90}, treating the equation $-\Delta_g u - \lambda u=g$ in the Riemannian case. First of all, by employing the classical strategy of Aronszajn, Krywicki and Szarski in \cite{AKS62}, see also \cite[Section 5]{Rul18} for a nice exposition, one can work wih geodesic polar coordinates as in the proof of \cite[Lemma A]{DF90}. This can be done even for Lipschitz metrics. Let us now present the key arguments to obtain \eqref{eq:secondCarleman} following line by line the proof of \cite[Lemma A]{DF90}.

\medskip

\noindent \textbf{Modifications in the proof of \cite[Lemma A]{DF90} to obtain \Cref{lm:Carlemanpunctured}.} All the computations are the same until the equation $(2\mathrm{b})$. Note that their parameter $\beta$ has the role of our parameter $\alpha$. The equation $(3)$ that now needs to be solved is the following one
$$ y'+ \mu y - y^2 = \varepsilon^{-2} \phi(\varepsilon^{-1}x).$$ 
One can then solve it for $0 \leq \varepsilon^{-1} x \leq C$, where $C>0$ is a positive constant. The remaining arguments are the same leading to the estimate of the first left hand side term of \eqref{eq:secondCarleman} by the right hand side term of \eqref{eq:secondCarleman}. To obtain the gradient term, we use a Cacciopoli's estimate for the equation $-\Delta_{A} f =g$ by multiplying it by $\varepsilon^{-2} \chi^2 \rho^{-2 \alpha} f$, where $\chi$ is a suitable cut-off function equal to $1$ in a $(r\varepsilon)$-neighborhood of $|x| = r$.

\medskip

In \Cref{sec:proofCarleman}, for the sake of clarity, we give a full self-contained proof of \Cref{lm:Carleman} for general Lipschitz metric $A$ and a simple proof of \Cref{lm:Carlemanpunctured} for the flat metric $A=\Delta$ that differs from \cite[Lemma A]{DF90}. As before, the general case can be obtained by constructing geodesic polar coordinates adapted to the metric $A$. The details are omitting to keep the things simple.


\subsection{Vanishing order estimate}

Before proving \Cref{tm:mainresult2}, one needs to prove a result on the vanishing order estimate for $A$-harmonic functions with bounded doubling index.

\medskip

First, we have the following modification of the Carleman estimate \eqref{eq:Carleman1} from \Cref{lm:Carleman}.

\begin{lm}
\label{lm:CarlemanModif}
There exists a positive constant $C=C(A)>0$ such that for every $x_0 \in B_{1/4}$, for every $\alpha \geq C$, $f \in C_c^{\infty}(B_{2} \setminus \{x_0\})$, the following estimate holds
\begin{equation}
\label{eq:CarlemanModif}
\alpha^3 \int_{B_{2}} |x-x_0|^{-1-2 \alpha} |f|^2 dx + \alpha  \int_{B_{2}} |x-x_0|^{1- 2 \alpha} |\nabla f|^2 dx  
\leq C  \int_{B_{2}}  |x-x_0|^{2-2 \alpha} |\Delta_A f|^2 dx.
\end{equation}
\end{lm}
The proof of \Cref{lm:CarlemanModif} is a small adaptation of \Cref{lm:Carleman}.

\medskip

As a consequence of \Cref{lm:CarlemanModif}, we deduce the following result.
\begin{lm}
\label{lm:vanishingorderestimate}
There exists $C=C(A)>0$ such that for every $u \in H_{\text{loc}}^1(B_2) \cap L^{\infty}(B_2)$ satisfying \eqref{eq:Aharmonic} and \eqref{eq:bounddoublingindex}, the following estimate holds
\begin{equation}
\label{eq:vnaishingBall}
\sup_{B(x,1/4)} |u| \geq \exp(-CN) \sup_{B_2} |u| \qquad \forall x \in B_{1/2}.
\end{equation}
\end{lm}
\begin{proof}
Let us fix $x_0 \in B(0,1/2)$. Let $x_{\max} \in B(0,1)$ be such that $$|u(x_{\max})| = \sup_{B(0,1)} |u|.$$ We distinguish two cases.

\medskip

 \textbf{First case: $x_{\max} \in B(x_0,1/4)$.} Then we have by the doubling property \eqref{eq:bounddoublingindex},
$$\sup_{B(x_0,1/4)} |u| = \sup_{B(0,1)} |u| \geq \exp(-N) \sup_{B_2} |u|,$$
so \eqref{eq:vnaishingBall} holds.

\medskip

 \textbf{Second case: $x_{\max} \notin B(x_0,1/4)$.} Let $\chi \in C_c^{\infty}(B_2;[0,1])$ be a cut-off function such that
\begin{align*}
\chi &= 0\ \text{in}\ B(x_0,1/16),\\
\chi &= 1\ \text{in}\  B(0,15/8) \setminus B(x_0,1/8),\\
\chi &= 0\ \text{in}\ B(0,2) \setminus B(0,31/16).
\end{align*}
In particular $\chi \equiv 1$ in $B(x_{\max}, 1/8)$. By local elliptic regularity, see \cite[Theorem 8.8]{GT01}, we have that $u \in H_{\text{loc}}^2(B_2)$ so by a straightforward density argument one can apply the modified version of the Carleman estimate \eqref{eq:CarlemanModif} of \Cref{lm:CarlemanModif} to $f := \chi u$. By denoting $w(x)  =|x-x_0|$, we then obtain
$$
\alpha^3 \int_{B_{2}} w^{-1-2 \alpha} |f|^2 dx 
\leq C  \int_{B_{2}}  w^{2-2 \alpha} |\Delta_A f|^2 dx.
$$
By using the equation \eqref{eq:Aharmonic}, we deduce that
\begin{equation}
\label{eq:proofeasy1}
\alpha^3 \int_{B_{2}} w^{-1-2 \alpha} \chi^2 |u|^2 dx 
\leq C  \int_{B_{2}}  w^{2-2 \alpha} (|\nabla \chi|^2 |\nabla u|^2 +[|\nabla \chi|^2 + |D^2 \chi|^2] |u|^2) dx.
\end{equation}
By using the definition of $\chi$ and De Giorgi-Nash-Moser theory, see \cite[Theorem 4.1]{HL11} to pass from a $L^2$-estimate to a $L^{\infty}$-estimate, one can bound from below the left hand side as follows
\begin{equation}
\label{eq:proofeasy2}
\alpha^3 \int_{B_{2}} w^{-1-2 \alpha} \chi^2 |u|^2 dx \geq \exp(-C_2 \alpha) |u(x_{\max})|^2,
\end{equation}
for some positive constant $C_2>0$. Then, by Cacciopoli's estimates, one gets that there exist positive constants $C_3 > C_2 > C_1 > 0$ and a positive constant $C=C(A)>0$ such that
\begin{multline}
\label{eq:proofeasy3}
\int_{B_{2}}  w^{2-2 \alpha} (|\nabla \chi|^2 |\nabla u|^2 +[|\nabla \chi|^2 + |D^2 \chi|^2] |u|^2) dx \\
\leq C \exp(-C_3 \alpha) (\sup_{B_2} |u|)^2 + C \exp(-C_1 \alpha)( \sup_{B(x_0,1/8)} |u| )^2.
\end{multline}
By using the definition of $x_{\max}$, the doubling index estimate \eqref{eq:bounddoublingindex}, \eqref{eq:proofeasy1},\eqref{eq:proofeasy2}, \eqref{eq:proofeasy3} we therefore obtain that
$$ \exp(-C_2 \alpha) (\sup_{B_1} |u|)^2 \leq C \exp(-C_3 \alpha) \exp(2 N) (\sup_{B_1} |u|)^2 + C \exp(-C_1 \alpha)( \sup_{B(x_0,1/8)} |u| )^2.$$
Now the punchline, by taking $\alpha \geq C(C_2,C_3,A) N$, the first right hand side term can be hidden in the left hand side term because $C_3 > C_2$ to deduce that
$$  \sup_{B_1} |u| \leq \exp(C N) \sup_{B(x_0,1/8)} |u|.$$
This concludes the proof of \eqref{eq:vnaishingBall} recalling again \eqref{eq:bounddoublingindex}.
\end{proof}

\subsection{Sharp local $L^2$ estimates on annuli and consequences at $L^p$-level}

In this part, first we are going to prove the sharp $L^2$-estimate \eqref{eq:L2BernsteinfunctionAnnulusharmonic} of \Cref{tm:df90harmonic} following the original proof of $L^2$-Bernstein estimates for Laplace eigenfunctions from \cite{DF90}. Then, we deduce that the doubling index of $A$-harmonic functions on annulus of width comparable to the wavelength at $L^p$-level growth at most polynomially.
\begin{pr}\label{pr:doublingpolynomial}
There exist $r_0, C >0$ depending only on $A$, such for every $p \in [1, \infty]$, for every $A$-harmonic function $u \in H_{\text{loc}}^1(B_2) \cap L^{\infty}(B_2)$, i.e. satisfying \eqref{eq:Aharmonic} with a bounded doubling index $N$ defined in \eqref{eq:bounddoublingindex}, for every $r \in (0,r_0)$,
\begin{align}
\label{eq:LpBernsteinfunctionAnnulusharmonicFirst}
 \|u\|_{L^p\left(r < |x| < r\left(1+\frac{2}{N}\right)\right)} &\leq C N^{\beta} \|u\|_{L^p\left(r < |x| < r\left(1+\frac{1}{N}\right)\right)}, \qquad \beta = \frac{d-1}{2}.
 \end{align}
\end{pr}

\begin{proof}[Proofs of \eqref{eq:L2BernsteinfunctionAnnulus} of \Cref{tm:df90harmonic} and \Cref{pr:doublingpolynomial}]
We split the proof into several steps.

\medskip

 \textbf{Step 1: Carleman estimate to a truncated version of $u$.}
Let us take 
$\chi \in C_c^{\infty}(B_2)$ a cut-off function such that
\begin{align*}
\chi &= 0 \quad \text{for } |x|\le r(1+(1/4)\alpha^{-1}),\\
\chi &= 1 \quad \text{for } r(1+(1/2)\alpha^{-1}) \le |x| \le 15/8,\\
\chi &= 0 \quad \text{for } 31/16 \le |x| \le 2,
\end{align*}
satisfying the estimates
\begin{equation}
\label{eq:variationcutoffboundary}
|\nabla \chi| \le C ,\quad |D^2 \chi| \le C, \qquad 15/8 \le |x| \le 31/16,
\end{equation}
and
\begin{equation}
\label{eq:variationcutoffsmallball}
|\nabla \chi| \le C r^{-1} \alpha,\quad |D^2 \chi| \le C r^{-2} \alpha^{-2},\qquad 
r(1+(1/4) \alpha^{-1}) \le |x| \le r(1+(1/2) \alpha^{-1}).
\end{equation}

We define $f = \chi u$, note that $f \in H_{\text{loc}}^2(B_2)$ with $\text{supp}(f) \subset\subset \{r \le |x|\le 2\}$, so we can apply the Carleman estimates \eqref{eq:Carleman1} from Lemma \ref{lm:Carleman} and \eqref{eq:secondCarleman} from \Cref{lm:Carlemanpunctured} to $f$, taking $\varepsilon = \alpha^{-1}$ to get
\begin{equation}
\label{eq:secondCarlemanProof}
\alpha^3 \int_{|x|\le 2} \rho^{-1-2 \alpha} |f|^2 dx 
+  \frac{\alpha^4}{r^2} \int_{|x|\le r(1+2\alpha^{-1})} \rho^{-2 \alpha} |f|^2 dx 
\le C  \int_{|x|\le 2} \rho^{2-2 \alpha} |\Delta_A f|^2 dx.
\end{equation}
Using the elliptic equation \eqref{eq:Aharmonic} satisfied by $u$, we deduce from \eqref{eq:secondCarlemanProof}
\begin{equation}
\label{eq:FinCarlemanPunctured}
\alpha^3 \int_{|x|\le 2} \rho^{-1-2 \alpha} \chi^2 |u|^2 dx  
+ \frac{\alpha^4}{r^2} \int_{|x|\le r(1+\alpha^{-1})} \rho^{-2 \alpha} \chi^2 |u|^2dx 
\le C (I_1 +  I_2),
\end{equation}
where
\begin{equation}
\label{eq:defI1}
I_1 = \int_{r(1+(1/4) \alpha^{-1}) \le |x| \le r(1+(1/2) \alpha^{-1})}  
\rho^{2-2 \alpha} (|\nabla \chi|^2 |\nabla u|^2 + [|\nabla \chi|^2 + |D^2 \chi|^2] |u|^2) dx,
\end{equation}
and
\begin{equation}
\label{eq:defI2}
I_2 = \int_{15/8 \le |x| \le 31/16}  
\rho^{2-2 \alpha} (|\nabla \chi|^2 |\nabla u|^2 + [|\nabla \chi|^2 + |D^2 \chi|^2] |u|^2) dx.
\end{equation}

\medskip
\textbf{Step 2: Absorption of the boundary terms.}
The goal of this step is to absorb the cut-off terms from \eqref{eq:FinCarlemanPunctured} that are located near the boundary $|x|=2$, that is the term $I_2$ defined in \eqref{eq:defI2}. By \eqref{eq:variationcutoffboundary} and from local elliptic regularity estimates, there exist constants $C_3>0$ and $C=C(A)>0$ such that
\begin{equation}
\label{eq:estimateI2}
I_2 \le C e^{-C_3 \alpha} (\sup_{|x|\le 2} |u|)^2.
\end{equation}
Moreover, by the definition of $\chi$ and local elliptic regularity, one gets
\begin{equation}
\label{eq:lowerboundLHS}
\alpha^3 \int_{|x|\le 2} \rho^{-1-2 \alpha} \chi^2 |u|^2 dx 
\ge C^{-1} e^{-C_2 \alpha} (\sup_{|x-x_0|\le 1/4} |u|)^2,
\end{equation}
where $0 < C_2 < C_3$ and $x_0 \in \{|x|\le 2\}$ is such that $|x_0| = 3/8$.  
Using the vanishing order estimate \eqref{eq:vnaishingBall} from \Cref{lm:vanishingorderestimate} (note that $\chi \equiv 1$ in $|x-x_0|\le 1/4$ for $r$ small enough), we deduce
\begin{equation}
\label{eq:lowerboundLHS2}
\alpha^3 \int_{|x|\le 2} \rho^{-1-2 \alpha} \chi^2 |u|^2 dx 
\ge C^{-1} e^{-C_2 \alpha} e^{-CN} (\sup_{|x|\le 2} |u|)^2.
\end{equation}
By taking $\alpha \ge C(C_2,C_3,A) N$, we obtain from \eqref{eq:FinCarlemanPunctured}, \eqref{eq:estimateI2}, and \eqref{eq:lowerboundLHS2} that $I_2$ can be absorbed into the left-hand side
\begin{equation}
\label{eq:FinCarlemanPuncturedAbsorption}
\alpha^3 \int_{|x|\le 2} \rho^{-1-2 \alpha} \chi^2 |u|^2 dx  
+ \frac{\alpha^4}{r^2} \int_{|x|\le r(1+2\alpha^{-1})} \rho^{-2 \alpha} \chi^2 |u|^2 dx  
\le C I_1.
\end{equation}

\medskip
 \textbf{Step 3: Growth estimate at $L^2$-regularity.}
We now compare the second left-hand side term in \eqref{eq:FinCarlemanPuncturedAbsorption} and the right-hand side term to obtain the growth estimate. Since $\rho^{-2\alpha}$ has comparable amplitude for $r \le |x| \le r(1+2\alpha^{-1})$, we can simplify \eqref{eq:FinCarlemanPuncturedAbsorption} using \eqref{eq:defI1}
\begin{multline}
\label{eq:FinCarlemanPuncturedAbsorptionSimp}
\frac{\alpha^4}{r^2} \int_{|x|\le r(1+2\alpha^{-1})} \chi^2 |u|^2 dx \\
 \le C \int_{r(1+(1/4)\alpha^{-1}) \le |x| \le r(1+(1/2)\alpha^{-1})}  
r^2 (|\nabla \chi|^2 |\nabla u|^2 + [|\nabla \chi|^2 + |D^2 \chi|^2] |u|^2) dx.
\end{multline}
Using \eqref{eq:variationcutoffsmallball}, we deduce
\begin{multline}
\label{eq:FinCarlemanPuncturedAbsorptionSimp2}
\frac{\alpha^4}{r^2} \int_{|x|\le r(1+2\alpha^{-1})} \chi^2 |u|^2 dx \\
\le C \int_{r(1+(1/4)\alpha^{-1}) \le |x| \le r(1+(1/2)\alpha^{-1})}  
r^2 ((r^{-1}\alpha)^2 |\nabla u|^2 + (r^{-2}\alpha^2)^2 |u|^2) dx.
\end{multline}
By applying Cacciopoli's inequality on every ball $B(x_0,C_4 r \alpha^{-1}) \subset \subset B(x_0, C_5 r \alpha^{-1})$ for suitable constants $0 < C_4 < C_5$,
 we have \begin{equation} \int_{B(x_0,C_4 r \alpha^{-1})} |\nabla u|^2 dx \leq C \frac{\alpha^2}{r^2} \int_{B(x_0,C_5 r \alpha^{-1})} |u|^2) dx, \qquad \forall x_0 \in B_2,
\end{equation} then by taking a suitable covering of the annulus $\{r(1+(1/4)\alpha^{-1}) \le |x| \le r(1+(1/2)\alpha^{-1})\}$ by balls of radius $C_4 r \alpha^{-1}$, we get for some $C_0,C_0'>0$ such that $C_0 < 1/4 < 1/2 < C_0' < 3/4$, \begin{equation} \int_{r(1+(1/4)\alpha^{-1}) \le |x| \le r(1+(1/2)\alpha^{-1})} |\nabla u|^2 \leq C \frac{\alpha^2}{r^2} \int_{r(1+C_0\alpha^{-1}) \le |x| \le r(1+C_0'\alpha^{-1})} |u|^2) dx, 
\end{equation}
so from \eqref{eq:FinCarlemanPuncturedAbsorptionSimp2}, we deduce
\begin{equation}
\label{eq:FinCarlemanPuncturedAbsorptionSimp4}
\int_{r(1+2^{-1}\alpha^{-1}) \le |x| \le r(1+2\alpha^{-1})} |u|^2 dx 
\le C \int_{r(1+C_0\alpha^{-1}) \le |x| \le r(1+C_0'\alpha^{-1})} |u|^2 dx.
\end{equation}
By adding $\|u\|_{L^2(r \le |x| \le r(1+\alpha^{-1}))}^2$ on both sides,
\begin{equation}
\label{eq:FinCarlemanPuncturedAbsorptionSimp5}
\int_{r \le |x| \le r(1+2\alpha^{-1})} |u|^2 dx 
\le C \int_{r \le |x| \le r(1+\alpha^{-1})} |u|^2 dx.
\end{equation}
Recalling $\alpha \ge C(C_2,C_3,A)N$, and by iterating \eqref{eq:FinCarlemanPuncturedAbsorptionSimp5} a finite number of times depending on $C(C_2,C_3,A)$,  this gives \eqref{eq:L2BernsteinfunctionAnnulusharmonic}.

\medskip
 \textbf{Step 4: Growth estimate at $L^p$-regularity.}
For $p\in[2,\infty]$ (the case $p\in[1,2)$ is similar), by De Giorgi-Nash-Moser theory, see \cite[Theorem 4.1]{HL11} to pass from a $L^2$-estimate to a $L^{\infty}$-estimate,
\begin{equation}
\label{eq:Boot10}
\|u\|_{L^\infty(|x-x_0|\le C_6r\alpha^{-1})} 
\le C (\alpha r^{-1})^{d/2} \|u\|_{L^2(|x-x_0|\le C_7r\alpha^{-1})},
\qquad r < |x_0| < 2r.
\end{equation}
Then, from \eqref{eq:FinCarlemanPuncturedAbsorptionSimp4} and \eqref{eq:Boot10}, we deduce
\begin{equation}
\label{eq:Boot1infty}
\|u\|_{L^\infty(r(1+C_2\alpha^{-1}) \le |x| \le r(1+C_3\alpha^{-1}))}
\le C (\alpha r^{-1})^{d/2} 
\|u\|_{L^2(r(1+C_0\alpha^{-1}) \le |x| \le r(1+C_0'\alpha^{-1}))}.
\end{equation}
Using Hölder’s inequality and the volume bound $|r(1+C_0\alpha^{-1}) \le |x| \le r(1+C_0'\alpha^{-1}))| \le C r^d \alpha^{-1}$, we find
\begin{equation}
\label{eq:Boot3ter}
\|u\|_{L^p(r(1+C_2\alpha^{-1}) \le |x| \le r(1+C_3\alpha^{-1}))}
\le C \alpha^{\beta}
\|u\|_{L^p(r(1+C_0\alpha^{-1}) \le |x| \le r(1+C_0'\alpha^{-1}))}.
\end{equation}
Recalling $\alpha \ge C N$, we obtain
\begin{equation}
\label{eq:Boot3four}
\|u\|_{L^p(r(1+C_2\alpha^{-1}) \le |x| \le r(1+C_3\alpha^{-1}))}
\le C N^{\beta}
\|u\|_{L^p(r(1+C_0\alpha^{-1}) \le |x| \le r(1+C_0'\alpha^{-1}))}.
\end{equation}
By adding $\|u\|_{L^p(r \le |x| \le r(1+\alpha^{-1}))}$ on both sides, we finally obtain
\begin{equation}
\label{eq:FinCarlemanPuncturedAbsorptionSimp5Boot}
\|u\|_{L^p(r \le |x| \le r(1+2\alpha^{-1}))}
\le C N^{\beta}
\|u\|_{L^p(r \le |x| \le r(1+\alpha^{-1}))},
\end{equation}
which concludes the proof of \eqref{eq:LpBernsteinfunctionAnnulusharmonicFirst} by iterating \eqref{eq:FinCarlemanPuncturedAbsorptionSimp5Boot} a finite number of times.
\end{proof}

In the following, we will rather use the following estimate.
\begin{pr}
\label{pr:doublingpolynomialPratical}
For every $0<C_1<C_2<C_3<C_4$, there exist $r_0, C >0$ depending only on $A$, $C_1$, $C_2$, $C_3$, $C_4$, such that for every $p \in [1, \infty]$,  for every $A$-harmonic function $u \in H_{\text{loc}}^1(B_2) \cap L^{\infty}(B_2)$, i.e. satisfying \eqref{eq:Aharmonic} with a bounded doubling index $N$ defined in \eqref{eq:bounddoublingindex}, for every $r \in (0,r_0)$,
\begin{align}
 & \|u\|_{L^p\left(r\left(1+C_3\frac{\log(N)}{N}\right) < |x| < r\left(1+C_4\frac{\log(N)}{N}\right)\right)} \notag\\
  &\leq C N^{\beta} \|u\|_{L^p\left(r\left(1+C_1\frac{ \log(N)}{N}\right) < |x| < r\left(1+C_2\frac{\log(N)}{N}\right)\right)}, \qquad \beta =C(d, C_1, C_2, C_3, C_4)>0. \label{eq:LpBernsteinfunctionAnnulusharmonicSecondPratical}
 \end{align}
\end{pr}

\begin{proof}
The proof is similar to the proof of \Cref{pr:doublingpolynomial}. We follow it by only sketching it and pointing out the main differences.

\medskip

We first set $\varepsilon = \log(N)/N$. Starting from a cut-off version of $u$ vanishing in a $(r\varepsilon$)-neighborhood of $\{|x|\le r\}$, we apply the Carleman estimates \eqref{eq:Carleman1} from Lemma \ref{lm:Carleman}, and \eqref{eq:secondCarleman} from \Cref{lm:Carlemanpunctured}. The absorption of the boundary terms as in Step 2 remains unchanged, taking $\alpha \ge C N$.

\medskip 

For the growth estimate at $L^2$-regularity (i.e. Step 3), the difference is that the weight $\rho^{-2\alpha}$ does not have the same amplitude in the region $r \le |x| \le r(1+2\varepsilon)$, so one cannot simplify on both sides of the estimates. However, by using $\alpha \varepsilon = C \log(N)$, we obtain that \eqref{eq:FinCarlemanPuncturedAbsorptionSimp} becomes
\begin{multline}
\label{eq:FinCarlemanPuncturedAbsorptionSimpBis}
r^2 \int_{|x|\le r(1+2\varepsilon)} \chi^2 |u|^2 dx \\
\le C N^{C} \int_{r(1+(1/4)\varepsilon) \le |x| \le r(1+(1/2)\varepsilon)} 
r^2 (|\nabla \chi|^2 |\nabla u|^2 + [|\nabla \chi|^2 + |D^2 \chi|^2] |u|^2) dx.
\end{multline}
Then, at $L^2$-regularity, we first obtain 
\begin{equation}
\label{eq:FinCarlemanPuncturedAbsorptionSimp5Bis}
\int_{r \le |x| \le r(1+2\varepsilon)} |u|^2 dx 
\le C N^C \int_{r \le |x| \le r(1+\varepsilon)} |u|^2 dx.
\end{equation}
By employing local elliptic regularity estimates as in Step 4, we finally get \eqref{eq:LpBernsteinfunctionAnnulusharmonicSecondPratical}.
\end{proof}

\subsection{Almost sharp local $L^p$ estimates on annuli}
\label{sec:almostsharp}

This part is devoted to the proof of \eqref{eq:LpBernsteinfunctionAnnulusharmonic} of \Cref{tm:mainresult2}. The other estimates \eqref{eq:LpBernsteinfunctionharmonic}, \eqref{eq:LpBernsteingradientharmonic} are straightforward consequences, see \Cref{sec:endproofharmonic} below.

\medskip

In the next, we use the notation for the annulus
$$ A_{r}^{(\gamma_1, \gamma_2)} = \{x\in\mathbb{R}^d:\; r(1+\gamma_1)<|x|<r(1+\gamma)\}\qquad r>0, \ \gamma_2 >\gamma_1 \geq  0 .$$

We start with the following covering lemma.
\begin{lm}
\label{lem:coveringannulus}
Let \(r>0\) and \(0<\varepsilon<1\). There exists a finite cover of
\(A_r^{(0,2 \varepsilon)}\) by \(N\) balls \(B(x_i,4r\varepsilon)\) such that
each \(8r\varepsilon\)-ball meets at most \(C\) other \(8r\varepsilon\)-balls,
where \(C=C(d)=9^d>0\).
\end{lm}

\begin{proof}
Choose a maximal (with respect to inclusion) finite set of centers
\(\{x_i\}_{i=1}^N \subset A_r^{(0,2 \varepsilon)}\) for which the closed balls
\(\overline{B(x_i,2r\varepsilon)}\) are pairwise disjoint. Such a maximal
finite collection exists because \(A_r^{(0,2 \varepsilon)}\) is bounded. By
maximality, for every \(y \in A_r^{(0,2 \varepsilon)}\) the ball
\(B(y,2r\varepsilon)\) meets some \(B(x_i,2r\varepsilon)\). Hence
\(y \in B(x_i,4r\varepsilon)\) for some \(i\), so
\(\{B(x_i,4r\varepsilon)\}_{i=1}^N\) covers \(A_r^{(0,2 \varepsilon)}\).

\medskip

Fix \(i\). If \(B(x_j,8r\varepsilon)\) intersects \(B(x_i,8r\varepsilon)\),
then \(|x_j-x_i|<16r\varepsilon\). Each corresponding disjoint ball
\(B(x_j,2r\varepsilon)\) is contained in
\(B(x_i,16r\varepsilon+2r\varepsilon)=B(x_i,18r\varepsilon)\).
The number of such disjoint balls is bounded by the volume ratio
\[
\frac{\operatorname{Vol}(B(0,18r\varepsilon))}{\operatorname{Vol}(B(0,2r\varepsilon))} = 9^d.
\]
Hence any given \(B(x_i,8r\varepsilon)\) can intersect at most \(9^d\)
other balls \(B(x_j,8r\varepsilon)\). 
\end{proof}

\begin{figure}[h!]
\centering
\begin{tikzpicture}[scale=0.8]
  \pgfmathsetmacro{\r}{3}
  \pgfmathsetmacro{\eps}{0.12}
  \pgfmathsetmacro{\inner}{\r}
  \pgfmathsetmacro{\outer}{\r*(1+2*\eps)}
  \pgfmathsetmacro{\ballrad}{4*\r*\eps}

  \foreach \ang/\radmult in {
    0/1.02, 30/1.06, 60/1.10, 90/1.16, 120/1.08, 150/1.13,
    180/1.05, 210/1.18, 240/1.12, 270/1.04, 300/1.15, 330/1.09
  }{
    \coordinate (c\ang) at ({\ang}:\radmult*\r cm);
    \draw[fill=blue!25,draw=blue!60,semithick] (c\ang) circle (\ballrad cm);
  }
  
  \draw[thick] (0,0) circle (\inner cm);
  \draw[thick] (0,0) circle (\outer cm);


\end{tikzpicture}
\caption{Covering of the annulus $A_r^{(0,2 \varepsilon)}$ by balls $B(x_i,4 r \varepsilon)$}
\end{figure}

\begin{cor}
\label{cro:eqnorms}
For every \(p \in [1,\infty)\), for every $u \in L^p(A_r^{(0,2 \varepsilon)})$, we have
\begin{equation}
\label{eq:equivalencenorms}
C^{-1} \sum_{i=1}^N \|u\|_{L^p(B(x_i, 8 r \varepsilon) \cap A_r^{(0,2 \varepsilon)})}^p
\le \|u\|_{L^p(A_r^{(0,2 \varepsilon)})}^p
\le \sum_{i=1}^N \|u\|_{L^p(B(x_i, 4 r \varepsilon) \cap A_r^{(0,2 \varepsilon)})}^p,
\end{equation}
for the same constant \(C=C(d)>0\) as in \Cref{lem:coveringannulus}.
\end{cor}

\begin{proof}
Since \(\{B(x_i,4r\varepsilon)\}_{i=1}^N\) covers \(A_r^{(0,2 \varepsilon)}\),
\[
\|u\|_{L^p(A_r^{(0,2 \varepsilon)})}^p
= \int_{A_r^{(0,2 \varepsilon)}} |u|^p \, dx
= \int_{A_r^{(0,2 \varepsilon)} \cap (\cup_i B(x_i,4r\varepsilon))} |u|^p \, dx
\le \sum_{i=1}^N \int_{A_r^{(0,2 \varepsilon)} \cap B(x_i,4r\varepsilon)} |u|^p \, dx,
\]
which gives the right-hand inequality in \eqref{eq:equivalencenorms}.

\medskip

For the left-hand inequality, note that
\[
\sum_{i=1}^N \|u\|_{L^p(B(x_i,8r\varepsilon) \cap A_r^{(0,2 \varepsilon)})}^p
= \sum_{i=1}^N \int_{A_r^{(0,2 \varepsilon)}} |u|^p \, 1_{B(x_i,8r\varepsilon)} \, dx
= \int_{A_r^{(0,2 \varepsilon)}} \Big( \sum_{i=1}^N 1_{B(x_i,8r\varepsilon)} \Big) |u|^p \, dx.
\]
By \Cref{lem:coveringannulus}, the overlap function
\(\sum_i 1_{B(x_i,8r\varepsilon)} \le C=C(d)\). Therefore,
\[
\sum_{i=1}^N \|u\|_{L^p(B(x_i,8r\varepsilon) \cap A_r^{(0,2 \varepsilon)})}^p
\le C \int_{A_r^{(0,2 \varepsilon)}} |u|^p \, dx
= C \|u\|_{L^p(A_r^{(0,2 \varepsilon)})}^p,
\]
which yields the left-hand inequality in \eqref{eq:equivalencenorms}.
\end{proof}

We can now pass to the proof of \eqref{eq:LpBernsteinfunctionAnnulusharmonic} of \Cref{tm:mainresult2}.
\begin{proof}[Proof of \eqref{eq:LpBernsteinfunctionAnnulusharmonic} of \Cref{tm:mainresult2}]
We consider the case \(p \in (2,\infty)\); the case \(p \in [1,2)\) follows by similar arguments, while the case \(p = \infty\) is discussed at the end of the proof.  

\medskip

\textbf{Step 0: Parameters and cut-off functions.} We begin by introducing the following free parameters, to be fixed later in the proof:
\begin{equation}
\alpha \ge C, \quad \varepsilon \in (0,c), \quad \delta \in (0,c).
\end{equation}

We apply \Cref{lem:coveringannulus} to cover \(A_r^{(0,2\varepsilon)}\) by \(N\) balls of radius \(4 r \varepsilon\) such that each \(8 r \varepsilon\)-ball intersects at most \(C\) other \(8 r \varepsilon\)-balls, where \(C\) depends only on the dimension \(d\). \medskip 

We then introduce two types of cut-off functions: a radial, global cut-off function
\begin{equation}
\label{eq:defchialpha}
\chi_{\alpha}(x) =
\begin{cases}
0, & |x| \le r(1+\alpha^{-1}),\\[1ex]
1, & |x| \ge r(1+2 \alpha^{-1}),
\end{cases}
\end{equation}
satisfying 
\begin{equation}
    \label{eq:variationcutoffchialpha}
    |\nabla \chi_\alpha| \leq  \frac{C \alpha}{r}, \qquad
    |D^2 \chi_\alpha| \leq  \frac{C \alpha^2}{r^2}, \qquad 
    r(1+\alpha^{-1}) \leq |x| \le r(1+2\alpha^{-1}),
    \end{equation}
and a local cut-off on each ball \(B(x_i,8r\varepsilon)\)
\begin{equation}
\label{eq:defcutoffsmallball}
\tilde{\chi}_i(x) =
\begin{cases}
1, & |x-x_i| \le 4 r \varepsilon,\\[1ex]
0, & |x-x_i| \ge 8 r \varepsilon,
\end{cases}
\end{equation}
satisfying
\begin{equation}
    \label{eq:variationcutoffsmallballepsilon}
    |\nabla \tilde{\chi}_i| \leq  \frac{C}{r \varepsilon}, \qquad
    |D^2 \tilde{\chi}_i| \leq \frac{C}{r^2 \varepsilon^2}, \qquad 
    4 r \varepsilon \le |x-x_i| \le 8 r \varepsilon.
    \end{equation}
Finally, the composite cut-off on each ball is defined by
\begin{equation}
\chi_i(x) := \chi_\alpha(x)\, \tilde{\chi}_i(x).
\end{equation}

\begin{figure}[h!]
\centering
\begin{tikzpicture}[scale=1.2]

  \fill (0,0) circle (1pt) node[below left] {$0$};

  \draw[thick] (0,0) circle (2);   
  \draw[thick] (0,0) circle (2.3); 

  \coordinate (xi) at (2.0,0.8);
  \fill (xi) circle (2pt) node[right] {$x_i$};

  \draw[thick,green!70!black] (xi) circle (0.8);   
  \draw[thick,green!70!black,dashed] (xi) circle (1.6); 

  \begin{scope}
    \clip (xi) circle (1.6);              
    \fill[green!40,opacity=0.6] (0,0) circle (5); 
    \fill[white] (0,0) circle (2);        
  \end{scope}


\end{tikzpicture}
\caption{Support of $\chi_i = \chi_\alpha \tilde{\chi}_i$}
\end{figure}

\medskip

We split the proof into several steps.\medskip

 \textbf{Step 1: A $L^2$-Carleman estimate on each ball.} We apply two times the Carleman estimates \eqref{eq:secondCarleman} from \Cref{lm:Carlemanpunctured} to $\chi_i u$, where the small parameters are $\varepsilon$ and $\delta$, by taking
\begin{equation}
\label{eq:firstestimatealpha}
\alpha \geq C \varepsilon^{-1}, \ \alpha \geq C \delta^{-1}, 
\end{equation}
to obtain
\begin{multline}
\label{eq:applcarlemansmallball}
\frac{\alpha^2}{\varepsilon^2 r^2} \int_{|x| \leq r (1+2 \varepsilon)} \rho^{-2 \alpha} |\chi_i u|^2 dx +   \frac{1}{\varepsilon^2} \int_{|x| \leq r (1+2 \varepsilon)} \rho^{-2 \alpha} |\nabla (\chi_i u)|^2 dx \\
+ \frac{\alpha^2}{\delta^2 r^2} \int_{|x| \leq r (1+2 \delta)} \rho^{-2 \alpha} |\chi_i u|^2 dx 
+  \frac{1}{\delta^2} \int_{|x| \leq r (1+2 \delta)} \rho^{-2 \alpha} |\nabla (\chi_i u)|^2 dx \\
\leq  C  \int_{|x| \leq 2} \rho^{2-2 \alpha} |\Delta_A (\chi_i u)|^2 dx.
\end{multline}
By using the equation satisfied by $u$, one can develop the right hand side of \eqref{eq:applcarlemansmallball}. The cut-off region is divided into two parts. The first one, denoted by
\(\mathcal{O}_i\), corresponds to the region where \(\chi_{\alpha}\) varies, namely
\[
r\left(1+\alpha^{-1}\right) \leq |x| \leq r\left(1+2\alpha^{-1}\right).
\]
The second part, denoted by \(\mathcal{R}_i\), corresponds to the region where
\(\tilde{\chi}_i\) varies, that is,
\[
r\left(1+2\alpha^{-1}\right) \leq |x|
\quad \text{and} \quad
4r\varepsilon \leq |x-x_i| \leq 8r\varepsilon.
\]
More precisely by using \eqref{eq:defcutoffsmallball}, \eqref{eq:defchialpha}, \eqref{eq:variationcutoffchialpha}, \eqref{eq:variationcutoffsmallballepsilon}, \eqref{eq:firstestimatealpha} and the condition
\begin{equation}
\label{eq:firstestimatealphaSecond}
\delta \leq c \varepsilon, 
\end{equation}
we obtain
\begin{multline*}
 \int_{|x| \leq 2} \rho^{2-2 \alpha} |\Delta_A (\chi_i u)|^2 dx\\ 
\leq C \int_{\{r(1+\alpha^{-1}) \leq |x| \leq r(1+2 \alpha^{-1})\} \cap \{ |x-x_i|\leq 8 r \varepsilon\}} \rho^{2-2 \alpha} \left(r^{-2} \alpha^2 |\nabla u|^2 + r^{-4} \alpha^4 |u|^2 dx\right),\\
  + C \int_{\{r(1+2\alpha^{-1})) \leq |x|\} \cap \{4 r \varepsilon \leq |x-x_i|\leq 8 r \varepsilon\}}\rho^{2-2 \alpha} \left(r^{-2} \varepsilon^{-2} |\nabla u|^2 + r^{-4} \delta^{-4} |u|^2 dx\right)
\end{multline*}
so we have
\begin{equation}
 \int_{B_{2}} \rho^{2-2 \alpha} |\Delta_A (\chi_i u)|^2 dx  \leq C (\mathcal{O}_{i} + \mathcal{R}_{i}),\label{eq:applcarlemansmallballBisRHS}
\end{equation}
where 
\begin{align}
\mathcal{O}_{i} &= \int_{\{r(1+\alpha^{-1}) \leq |x| \leq r(1+2 \alpha^{-1})\} \cap \{ |x-x_i|\leq 8 r \varepsilon\}} \rho^{2-2 \alpha} \left(r^{-2} \alpha^2 |\nabla u|^2 + r^{-4} \alpha^4 |u|^2 dx\right),\label{eq:Oi}\\
\mathcal{R}_{i} &= \int_{\{r(1+2\alpha^{-1})) \leq |x|\} \cap \{4 r \varepsilon \leq |x-x_i|\leq 8 r \varepsilon\}}\rho^{2-2 \alpha} \left(r^{-2} \varepsilon^{-2} |\nabla u|^2 + r^{-4} \delta^{-4} |u|^2 dx\right).\label{eq:Ri}
\end{align}
We further decompose \(\mathcal{R}_i\) into two parts: an inner region, denoted by
\(\mathcal{R}_{i,\mathrm{in}}\), where \( |x| \leq r(1+2\delta) \), and an outer region,
denoted by \(\mathcal{R}_{i,\mathrm{out}}\). From \eqref{eq:applcarlemansmallball}, \eqref{eq:applcarlemansmallballBisRHS}, \eqref{eq:Ri} and \eqref{eq:firstestimatealphaSecond} we then have 
\begin{multline}
\label{eq:applcarlemansmallballCompactform}
\frac{\alpha^2}{\varepsilon^2 r^2} \int_{|x| \leq r (1+2 \varepsilon)} \rho^{-2 \alpha} |\chi_i u|^2 dx +   \frac{1}{\varepsilon^2} \int_{|x| \leq r (1+2 \varepsilon)} \rho^{-2 \alpha} \chi_i^2 |\nabla  u|^2 dx \\
+ \frac{\alpha^2}{\delta^2 r^2} \int_{|x| \leq r (1+2 \delta)} \rho^{-2 \alpha} |\chi_i u|^2 dx 
+  \frac{1}{\delta^2} \int_{|x| \leq r (1+2 \delta)} \rho^{-2 \alpha} \chi_i^2 |\nabla  u|^2 dx \\
\leq  C  (\mathcal{O}_{i} + \mathcal{R}_{i, \text{in}} + \mathcal{R}_{i, \text{out}}),
\end{multline}
where 
\begin{align}
\mathcal{R}_{i, \text{in}}&= \int_{\{r(1+2\alpha^{-1})) \leq |x| \leq r(1+2 \delta)\} \cap \{4 r \varepsilon \leq |x-x_i|\leq 8 r \varepsilon\}  }\rho^{2-2 \alpha} \left(r^{-2} \varepsilon^{-2} |\nabla u|^2 + r^{-4} \delta^{-4} |u|^2 dx\right) \label{eq:splittingI2iin}\\
\mathcal{R}_{i, \text{out}} &= \int_{\{r(1+2\delta) < |x|\} \cap \{4 r \varepsilon \leq |x-x_i|\leq 8 r \varepsilon\}  }\rho^{2-2 \alpha} \left(r^{-2} \varepsilon^{-2} |\nabla u|^2 + r^{-4} \delta^{-4} |u|^2 dx\right).\label{eq:splittingI2iout}
\end{align} 

 \textbf{Step 2: Summation of the estimates.} From  \eqref{eq:applcarlemansmallballCompactform}, we sum for $i=1$ to $i=N$ to obtain
\begin{multline}
\label{eq:applcarlemansmallballBisSum}
 \sum_{i=1}^N   \frac{\alpha^p}{\varepsilon^p r^p} \| \rho^{- \alpha} \chi_i u \|_{L^2(B(0,r(1+2 \varepsilon)))}^p 
 + \frac{1}{\varepsilon^p} \| \rho^{- \alpha} \chi_i \nabla u \|_{L^2(B(0,r(1+2 \varepsilon)))}^p\\+
 \sum_{i=1}^N   \frac{\alpha^p}{\delta^p r^p} \| \rho^{- \alpha} \chi_i u \|_{L^2(B(0,r(1+2 \delta)))}^p 
 + \frac{1}{\delta^p} \| \rho^{- \alpha} \chi_i \nabla u \|_{L^2(B(0,r(1+2 \delta)))}^p\\
 \leq C^p \left( \sum_{i=1}^N \mathcal{O}_{i} ^{p/2} + \mathcal{R}_{i, \text{in}}^{p/2} +\mathcal{R}_{i, \text{out}}^{p/2}\right).
\end{multline}
As we are going to see, the right hand side terms of \eqref{eq:applcarlemansmallballBisSum} will be treated in a rather different way. The first term $\mathcal{O}_{i}$ will contribute to our observation term, while the terms $\mathcal{R}_{i, \text{in}}$, $\mathcal{R}_{i, \text{out}}$ will be absorbed by using suitable conditions on the parameters $\alpha$, $\varepsilon$ and $\delta$, that would be compatible with \eqref{eq:firstestimatealpha} and \eqref{eq:firstestimatealphaSecond}.\medskip

 \textbf{Step 3: Absorption of the cut-off terms inside the annulus.} In this step, we focus on the second right hand side term of \eqref{eq:applcarlemansmallballBisSum}, i.e. $\mathcal{R}_{i, \text{in}}$ defined in \eqref{eq:splittingI2iin}. From \Cref{lem:coveringannulus}, we have
\begin{equation}
\label{eq:coveringapplication}
B(x_i, 8 r \varepsilon) \cap A_r^{(0,2 \delta)} \subset B(x_i, 8 r \varepsilon) \cap A_r^{(0,2 \varepsilon)} \subset \bigcup_{j \in V(i)} \{\tilde{\chi_j} = 1\},
\end{equation}
where $V(i)$ is the set of indices such that $B(x_j, 8 r \varepsilon) \cap B(x_i, 8 r \varepsilon) \neq \emptyset$. Moreover, we also have 
\begin{equation}
\label{eq:cardbounded}
\mathrm{Card}(V(i)) \leq C,
\end{equation} 
independent of $r$, $\varepsilon$.
By the successive application of \eqref{eq:coveringapplication}, \eqref{eq:cardbounded} the fact that $\tilde{\chi_j} = \chi_j$ in the zone $A_r^{(2\alpha^{-1}, 2\delta)} = \{(1+2 \alpha^{-1}) \leq |x| \leq r(1+2 \delta)\}$ and again \eqref{eq:cardbounded} we deduce  
\begin{align}
\notag
\sum_{i=1}^N \mathcal{R}_{i,\mathrm{in}}^{p/2} 
&\le C^p \sum_{i=1}^N 
\left(
\int_{A_r^{(2\alpha^{-1}, 2\delta)} \cap 
\bigcup\{\tilde{\chi}_j = 1\}}
\rho^{2-2\alpha}
\left( r^{-2}\varepsilon^{-2} |\nabla u|^2 
      + r^{-4}\delta^{-4} |u|^2 
\right)
\, dx
\right)^{p/2} \\
\notag
&\le C^p \sum_{i=1}^N \sum_{j\in V(i)}
\left(
\int_{A_r^{(2\alpha^{-1}, 2\delta)}
\cap \{\tilde{\chi}_j = 1\}}
\rho^{2-2\alpha}
\left( r^{-2}\varepsilon^{-2} |\nabla u|^2
      + r^{-4}\delta^{-4} |u|^2
\right)
\, dx
\right)^{p/2} \\
\notag
&\le C^p \sum_{i=1}^N \sum_{j\in V(i)}
\left(
\int_{A_r^{(2\alpha^{-1}, 2\delta)}
\cap \{\tilde{\chi}_j = 1\}}
\chi_j^2 \rho^{2-2\alpha}
\left( r^{-2}\varepsilon^{-2} |\nabla u|^2
      + r^{-4}\delta^{-4} |u|^2
\right)
\, dx
\right)^{p/2} \\
\notag
&\le C^p \sum_{i=1}^N 
\left(
\int_{A_r^{(2\alpha^{-1}, 2\delta)}}
\chi_i^2 \rho^{2-2\alpha}
\left( r^{-2}\varepsilon^{-2} |\nabla u|^2
      + r^{-4}\delta^{-4} |u|^2
\right)
\, dx
\right)^{p/2} \\
&\le C^p \sum_{i=1}^N 
\left(
\frac{1}{\delta^{2p} r^p}
\| \rho^{-\alpha}\chi_i u\|_{L^2(B(0,r(1+2\delta)))}^p
+
\frac{1}{\varepsilon^p}
\| \rho^{-\alpha}\chi_i \nabla u\|_{L^2(B(0,r(1+2\delta)))}^p
\right).
\label{eq:suminside}
\end{align}

So, from \eqref{eq:suminside}, one can absorb the second right hand side of \eqref{eq:applcarlemansmallballBisSum} by the last two left hand side terms of \eqref{eq:applcarlemansmallballBisSum} taking
\begin{equation}
\label{eq:deltaepsilonleq}
\alpha \geq C \varepsilon^{-1},\ \alpha \geq C\delta^{-1},\ \delta = c \varepsilon,
\end{equation}
for $c>0$ sufficiently small. So we have
\begin{multline}
\label{eq:applcarlemansmallballBisSumAbsorb1}
 \sum_{i=1}^N   \frac{\alpha^p}{\varepsilon^p r^p} \| \rho^{- \alpha} \chi_i u \|_{L^2(B(0,r(1+2 \varepsilon)))}^p 
 + \frac{1}{\varepsilon^p} \| \rho^{- \alpha} \chi_i \nabla u \|_{L^2(B(0,r(1+2 \varepsilon)))}^p\\
 \leq C^p \left( \sum_{i=1}^N \mathcal{O}_{i} ^{p/2} +\mathcal{R}_{i, \text{out}}^{p/2}\right).
\end{multline}

In the remaining part of the proof, we then fix $\delta$ as in \eqref{eq:deltaepsilonleq}, i.e. the constant $c>0$ is fixed. \medskip

 \textbf{Step 4: Elliptic regularity and bootstrap on each ball.}   Define the localized, weighted function on each ball by
\begin{equation}
\label{eq:vi}
v_{\alpha} := \rho^{-\alpha} \chi_\alpha u \quad \text{in } B(x_i,8 r \varepsilon).
\end{equation}
Then \(v_\alpha\) satisfies the elliptic equation
\begin{equation}
\label{eq:equationvi}
- \Delta_A v_\alpha = f_\alpha \quad \text{in } B(x_i,8 r \varepsilon),
\end{equation}
with source term
\begin{equation}
\label{eq:sourcetermvi}
|f_\alpha| \lesssim 
r^{-2} \alpha^2 \rho^{-\alpha} \chi_\alpha |u| 
+ r^{-1} \alpha \rho^{-\alpha} \big( |\nabla \chi_\alpha| |u| + \chi_\alpha |\nabla u| \big)
+ \rho^{-\alpha} \big( |D^2 \chi_\alpha| |u| + |\nabla \chi_\alpha| |\nabla u| \big).
\end{equation}
By interior \(W^{2,2}\)-estimates (cf. \cite[Theorem 9.11]{GT01}), the Sobolev embedding \(W^{2,2} \hookrightarrow W^{1,q}\) for some \(q=q(d)>2\), and a standard scaling argument, we obtain
\begin{multline}
\label{eq:sourcetermviEstimate}
\|v_\alpha\|_{L^q(B(x_i, C' r \varepsilon))} + r \varepsilon \|\nabla v_\alpha\|_{L^q(B(x_i, C' r \varepsilon))} \\
\lesssim (r \varepsilon)^{d/q - d/2} \Big[ r^2 \varepsilon^2 \|f_\alpha\|_{L^2(B(x_i,8 r \varepsilon))} + \|v_\alpha\|_{L^2(B(x_i,8 r \varepsilon))} \Big],
\end{multline}
for some \(4 < C' < 8\). Raising to the power \(p\) and summing over \(i=1,\dots,N\), and using \eqref{eq:applcarlemansmallballBisSumAbsorb1} and \eqref{eq:Oi}, we deduce the \(L^q\)-estimate
\begin{multline}
\label{eq:sourcetermviEstimateBis}
\sum_{i=1}^N \|v_\alpha\|_{L^q(B(x_i, C' r \varepsilon))}^p + (r \varepsilon)^p \|\nabla v_\alpha\|_{L^q(B(x_i,C' r \varepsilon))}^p\\
\leq C^p \big[ r (r \varepsilon)^{d/q - d/2} \alpha \varepsilon^3 \big]^p 
\sum_{i=1}^N \big( \mathcal{O}_i^{p/2} + \mathcal{R}_{i,\mathrm{out}}^{p/2} \big).
\end{multline}
A standard bootstrap argument using interior \(W^{2,q}\)-estimates (cf. \cite[Theorem 9.11]{GT01}) and Sobolev embeddings then yields, after finitely many steps (depending only on $d$) for some \(K=K(d)>0\)
\begin{equation}
\label{eq:sourcetermviEstimateTer}
\sum_{i=1}^N \|v_\alpha\|_{L^p(B(x_i, 4 r \varepsilon))}^p
\leq C^p \big[ r (r \varepsilon)^{d/p - d/2} \alpha^K \varepsilon^{2+K} \big]^p 
\sum_{i=1}^N \big( \mathcal{O}_i^{p/2} + \mathcal{R}_{i,\mathrm{out}}^{p/2} \big).
\end{equation}

%

 \textbf{Step 5: Absorption of the cut-off terms near the outer radius.}  We focus on the second term on the right-hand side of \eqref{eq:sourcetermviEstimateTer}.  From \eqref{eq:vi} and \Cref{cro:eqnorms}, we first remark that
\begin{equation}
\label{eq:sumouterradius}
\sum_{i=1}^N \|v_\alpha\|_{L^p(B(x_i,4 r \varepsilon))}^p
\geq \|\rho^{-\alpha} \chi_\alpha u\|_{L^p(A_r^{(\delta/2,\delta)})}^p
\geq \min_{|x| \leq r (1 + \delta)} \rho^{-p\alpha} \|\chi_\alpha u\|_{L^p(A_r^{(\delta/2,\delta)})}^p.
\end{equation}
On the other hand, using \eqref{eq:splittingI2iout}, \eqref{eq:deltaepsilonleq}, Hölder's inequality, elliptic regularity, and \eqref{eq:defchialpha}, we have
\begin{equation}
\sum_{i=1}^N \mathcal{R}_{i,\mathrm{out}}^{p/2} 
\leq C^p r^{-p} \varepsilon^{-2p} (r \varepsilon)^{pd/2 - d} \max_{|x| \geq r (1+2 \delta)} (\rho^{-p\alpha}) 
\sum_{i=1}^N \|\chi_\alpha u\|_{L^p(A_r^{(3\delta/2, C\varepsilon)})}^p.
\end{equation}
By \Cref{cro:eqnorms}, this simplifies to
\begin{equation}
\sum_{i=1}^N \mathcal{R}_{i,\mathrm{out}}^{p/2} 
\leq C^p r^{-p} \varepsilon^{-2p} (r \varepsilon)^{pd/2 - d} \max_{|x| \geq r (1+2 \delta)} (\rho^{-p\alpha}) 
\|\chi_\alpha u\|_{L^p(A_r^{(3\delta/2, C\varepsilon)})}^p.
\end{equation}
Choosing
\begin{equation}
\label{eq:definitionepsilon}
\varepsilon = c^{-1} \frac{\log N}{N},
\end{equation}
and using the doubling estimate \eqref{eq:LpBernsteinfunctionAnnulusharmonicSecondPratical} in Proposition \ref{pr:doublingpolynomialPratical}, we obtain
\begin{equation}
\label{eq:RHSouter}
\sum_{i=1}^N \mathcal{R}_{i,\mathrm{out}}^{p/2} 
\leq C^p r^{-p} \varepsilon^{-2p} (r \varepsilon)^{pd/2 - d} N^{p\beta} \max_{|x| \geq r (1+2 \delta)} (\rho^{-p\alpha}) 
\|\chi_\alpha u\|_{L^p(A_r^{(\delta/2, \delta)})}^p.
\end{equation}
To absorb \eqref{eq:RHSouter} into the left-hand side of \eqref{eq:sourcetermviEstimateTer}, using \eqref{eq:sumouterradius}, it suffices to choose \(\alpha\) and \(\varepsilon\) such that
\begin{equation}
e^{c p \alpha \varepsilon} \geq C^p (\alpha \varepsilon)^{pK} N^{p\beta},
\end{equation}
which is achieved by setting
\begin{equation}
\label{eq:alphaepsilon}
\alpha \geq C \frac{\log N}{\varepsilon} \geq C N.
\end{equation}
With this choice, we finally obtain
\begin{equation}
\label{eq:sourcetermviEstimateTerFinalAbsorb}
\sum_{i=1}^N \|v_\alpha\|_{L^p(B(x_i,4 r \varepsilon))}^p
\leq C^p [r (r \varepsilon)^{d/p - d/2} \alpha^K \varepsilon^{2+K}]^p \sum_{i=1}^N \mathcal{O}_i^{p/2}.
\end{equation}

From now on, \(\alpha\) is fixed according to \eqref{eq:alphaepsilon}.\medskip

 \textbf{Step 6: Estimate of the remaining cut-off terms \(\mathcal{O}_i\).} Using the embedding \(L^p(B(x_i, 8 r \varepsilon)) \hookrightarrow L^2(B(x_i, 8 r \varepsilon))\), Caccioppoli's inequality and \Cref{cro:eqnorms}, we first obtain
\begin{align}
\label{eq:estimateout}
\sum_{i=1}^N \mathcal{O}_{i}^{p/2}  \leq C^p r^{-p} \alpha^{2p} (r \varepsilon)^{pd/2 - d} \max_{r \leq |x| \leq r (1+ \alpha^{-1})}(\rho^{-\alpha}) \|  u\|_{L^p(A_{r}^{(0, 2\alpha^{-1})})}^p .
\end{align}
Combining \eqref{eq:estimateout}, \eqref{eq:sourcetermviEstimateTerFinalAbsorb}, and \Cref{cro:eqnorms}, we deduce
$$ \min_{r \leq |x| \leq r(1+4 \alpha^{-1})} (\rho^{-\alpha})\| \chi_{\alpha} u\|_{L^p(A_{r}^{(0, 2\alpha^{-1})})}^p \leq C^p (\alpha \varepsilon)^{(2+K)p} \max_{r \leq |x| \leq r (1+ \alpha^{-1})}(\rho^{-\alpha})\| \rho^{-\alpha} u\|_{L^p(A_{r}^{(0, \alpha^{-1})})}^p.$$
On the annulus scale $r(1+ \alpha^{-1}) \leq |x| \leq r(1+4 \alpha^{-1})$, we can divide both sides by \(\rho^{-\alpha}\), which gives
$$ \| u\|_{L^p(r(1+2 \alpha^{-1}) \leq |x| \leq r(1+4 \alpha^{-1}))} \leq C (\alpha \varepsilon)^{(2+K)} \|  u\|_{L^p(r(1+1 \alpha^{-1}) \leq |x| \leq r(1+2 \alpha^{-1}))}.$$
Finally, recalling \eqref{eq:definitionepsilon} and \eqref{eq:alphaepsilon}, this yields
\begin{equation}
\label{eq:finalestimateLp}
\| u\|_{L^p(r(1+2 \alpha^{-1}) \leq |x| \leq r(1+4 \alpha^{-1}))} \leq C (\log(N))^{(2+K)} \|  u\|_{L^p(r(1+ \alpha^{-1}) \leq |x| \leq r(1+2 \alpha^{-1}))}.
\end{equation}


\medskip

 \textbf{Step 7: Conclusion.} By iterating \eqref{eq:finalestimateLp} and recalling \eqref{eq:alphaepsilon} we finally obtain \eqref{eq:LpBernsteinfunctionAnnulusharmonic}.\medskip

 \textbf{Step 8: Proof of \eqref{eq:LpBernsteinfunctionAnnulusharmonic} in the $L^{\infty}$-case.} We can pass to the limit as $p \to \infty$ in \eqref{eq:finalestimateLp} because $C$ does not depend on $p$.
\end{proof}

\subsection{Proof of the growth estimates for $A$-harmonic functions}
\label{sec:endproofharmonic}

This (small) part is devoted to the proof of \eqref{eq:LpBernsteinfunctionharmonic} and \eqref{eq:LpBernsteingradientharmonic} from \Cref{tm:mainresult2}.

\begin{proof}[Proof of \eqref{eq:LpBernsteinfunctionharmonic} of \Cref{tm:mainresult2}]
A tiny adaptation of the proof \eqref{eq:LpBernsteinfunctionAnnulusharmonic} leads to
\begin{equation}
\label{eq:LpBernsteinfunctionAnnulusharmonic4}
\|u\|_{L^p\left(r' < |x| < r'\left(1+\frac{4}{N}\right)\right)} \leq C N^{\varepsilon} \|u\|_{L^p\left(r' < |x| < r'\left(1+\frac{1}{N}\right)\right)}\qquad \forall r' \in (0, r_0).
\end{equation}
We then apply \eqref{eq:LpBernsteinfunctionAnnulusharmonic4} with $r'=r(1-N)$, we obtain 
\begin{equation}
\|u\|_{L^p\left(r(1-\frac{1}{N}) < |x| < r(1-\frac{1}{N})\left(1+\frac{4}{N}\right)\right)} \leq C N^{\varepsilon} \|u\|_{L^p\left(r(1-\frac{1}{N}) < |x| < r(1-\frac{1}{N})\left(1+\frac{1}{N}\right)\right)}.
\end{equation}
We then have $r(1-\frac{1}{N})\left(1+\frac{4}{N}\right) \geq r(1+1/N)$, and $r(1-\frac{1}{N})\left(1+\frac{1}{N}\right) \leq r$ because $N \geq 2$. So this means that
\begin{equation}
\|u\|_{L^p\left(r(1-\frac{1}{N}) < |x| < r(1+\frac{1}{N})\right)} \leq C N^{\varepsilon} \|u\|_{L^p\left(r(1-\frac{1}{N}) < |x| < r\right)}.
\end{equation}
By adding $\|u\|_{L^p(B(0, r(1-1/N))}$ in both sides of the inequality, we obtain the inequality \eqref{eq:LpBernsteinfunctionharmonic}.
\end{proof}

\begin{proof}[Proof of \eqref{eq:LpBernsteingradientharmonic} of \Cref{tm:mainresult2}]
By a scaled local $W^{1,p}$ estimate applied to the equation \eqref{eq:Aharmonic}, see for instance \cite[Theorem 9.11]{GT01}, and by using \eqref{eq:LpBernsteinfunctionharmonic}, we have
$$ \|\nabla u\|_{L^p(B(x,r))} \leq C  \frac{N}{r} \| u\|_{L^p(B(x,r(1+N^{-1}))} \leq  C \frac{N^{1+\varepsilon}}{r} \| u\|_{L^p(B(x,r))} .$$
This ends the proof of \eqref{eq:LpBernsteingradientharmonic}.
\end{proof}

\subsection{Modification of the proofs for Laplace eigenfunctions}

The goal of this part is to prove \Cref{tm:mainresult1}.

\medskip

To simplify the notations, in all the following, the geodesic balls $B_g(x,r)$ will be denoted by $B(x,r)$ and the Riemannian volume of integration will be denoted by $dx$. The geodesic distance is denoted by $r(x) = d_g(x,x_0)$. \medskip

The following crucial $L^2$-Carleman estimates for the operator $-\Delta_g  - \lambda$  will be the main ingredients of the proof.
\begin{lm}
\label{lm:CarlemanLaplaceEigenStandard}
There exists a positive constant $C=C(M,g)>0$, $c=c(M,g)>0$, and an increasing function $\rho=\rho(r)$ for $0 < r < c$ satisfying
\begin{equation}
\label{eq:assumptionrhoMg}
C^{-1} \leq \frac{\rho(r)}{r} \leq C,\ C^{-1}  \leq |\partial_r \rho(r) | \leq C\ \qquad \forall r \in (0,c),
\end{equation}  
such that for every $r_M \in (0,c)$, $r \in (0,c)$, $x_0 \in M$, $\alpha \geq C (\sqrt{\lambda}+1)$, for all $f \in C_c^{\infty}(B_{r_M} \setminus \{x_0\})$, the following estimate holds
\begin{multline}
\label{eq:Carleman1LaplaceEigenStandard}
\alpha^3 \int_{B_g(x_0,r_M)} \rho^{-1-2 \alpha} |f|^2 dx + \alpha  \int_{B_g(x_0,r_M)} \rho^{1- 2 \alpha} |\nabla_g f|^2 dx  \\
\leq C  \int_{B_g(x_0,r_M)}  \rho^{2-2 \alpha} |\Delta_g f + \lambda f|^2 dx.
\end{multline}
\end{lm}
From now, $\rho$ is fixed as in \Cref{lm:CarlemanLaplaceEigenStandard}.
\begin{lm}
\label{lm:CarlemanpuncturedManifold}
There exists a positive constant $C=C(M,g)>0$ and $c=c(M,g)>0$ such that for every $r_M \in (0,c)$, $r \in (0,c)$, $x_0 \in M$, $\varepsilon \in (0,1)$,  $\alpha \geq C (\sqrt{\lambda}+\varepsilon^{-1})$, for all $f \in C_c^{\infty}(B_{r_M} \setminus B_r)$, the following estimate holds
\begin{multline}
\label{eq:secondCarlemananifold}
\frac{\alpha^2}{\varepsilon^2 r^2} \int_{B_g(x_0,r(1+2\varepsilon))} \rho^{-2 \alpha} |f|^2 dx + \frac{1}{\varepsilon^2} \int_{B_g(x_0,r(1+2\varepsilon))} \rho^{-2 \alpha} |\nabla_g f|^2 dx \\
\leq C  \int_{B_g(x_0,r_M)} \rho^{2-2 \alpha} |\Delta_g f + \lambda f|^2 dx.
\end{multline}
\end{lm}

The proofs of \Cref{lm:CarlemanLaplaceEigenStandard}, \Cref{lm:CarlemanpuncturedManifold} are similar to the ones of \Cref{lm:Carleman}, \Cref{lm:Carlemanpunctured}. The crucial difference between the proofs of \Cref{lm:Carleman}, \Cref{lm:Carlemanpunctured} is the presence of the zero-order term $\lambda f$ in the elliptic operator $-\Delta_g f  - \lambda f$ . This new term cannot be treated as a source term and absorbed by the Carleman parameter $\alpha$ because such a strategy would lead to \eqref{eq:secondCarlemananifold} for $\alpha \geq C (\lambda^{2/3}+1)$. This is why it has to be directly included in the symmetric part of the conjugated operator in the Carleman's strategy, note that here we crucially use the assumption that $\lambda$ is constant then does not depend on the $x$-variable. See \cite{DF88} and \cite{DF90} for details.

\medskip

The following classical vanishing order estimate holds for Laplace eigenfunctions.

\begin{lm}
\label{lm:vanishingorderestimateManifold}
There exist $C=C(M,g)>0$, $c=c(M,g)>0$ and $r_M, r_0 \in (0,c)$, such that for every $\lambda \geq 1$, for every $\varphi_{\lambda} \in C^{\infty}(M)$ satisfying \eqref{eq:Laplaceeingenfunction}, the following estimate holds
\begin{equation}
\label{eq:vnaishingBallManifold}
\sup_{B_g\left(x_0,\frac{r_0}{4}\right)} |\varphi_{\lambda}| \geq \exp(-C\sqrt{\lambda}) \sup_{B_g\left(x_0,r_M\right)} |\varphi_{\lambda}| \qquad \forall x \in B_g\left(x_0,\frac{r_0}{2}\right).
\end{equation}
\end{lm}
The proof of \Cref{lm:vanishingorderestimateManifold} is similar to the one of \Cref{lm:vanishingorderestimate} so we omit it.\medskip

With \Cref{lm:CarlemanpuncturedManifold} and \Cref{lm:vanishingorderestimateManifold}, we can first obtain the proof of \eqref{eq:L2BernsteinfunctionAnnulus}, then deduce a polynomial bound on growth estimates of Laplace eigenfunctions at the wavelength scale.
\begin{pr}\label{pr:doublingpolynomialLaplace}
There exist $C=C(M,g)$, $c=c(M,g)>0$ and $r_0 \in (0,c)$, such that for every $p \in [1, \infty]$, for every $\lambda \geq 1$, for every $\varphi_{\lambda} \in C^{\infty}(M)$ satisfying \eqref{eq:Laplaceeingenfunction}, for every $x \in M$, the following estimate 
\begin{align}
\label{eq:LpBernsteinfunctionAnnulusharmonicFirstLaplace}
 \|\varphi_\lambda(y)\|_{L^p\left(r < d_g(y,x) < r\left(1+\frac{2}{\sqrt{\lambda}}\right)\right)} &\leq C \lambda^{\beta} \|\varphi_\lambda(y)\|_{L^p\left(r < d_g(y,x) < r\left(1+\frac{1}{\sqrt{\lambda}}\right)\right)}, \qquad \beta = \frac{d-1}{4}.
 \end{align}
\end{pr}
\begin{rmk}
By combining \eqref{eq:LpBernsteinfunctionAnnulusharmonicFirstLaplace} with a local elliptic estimate, we obtain the following local $L^{\infty}$-Bernstein type estimate
\begin{equation}
\label{eq:weakLinftyBernsteingradientBetter}
\sup_{B_g\left(x,r\right)} |\nabla \varphi_{\lambda}| \leq C \frac{\lambda^{\frac{d+1}{4}}}{r} \sup_{B_g\left(x,r\right)} |\varphi_{\lambda}|.
\end{equation}
This estimate \eqref{eq:weakLinftyBernsteingradientBetter} is a little bit better than the previous Donnelly, Fefferman's estimate \eqref{eq:weakLinftyBernsteingradient}.
\end{rmk}
In the same spirit, we can obtain the following polynomial bound on growth estimates of Laplace eigenfunctions at a slightly bigger scale.
\begin{pr}
\label{pr:doublingpolynomialPraticalLaplace}
For every $0<C_1<C_2<C_3<C_4$, there exist $C>0$, $c>0$ and $r_0 \in (0,c)$ depending on $M$, $g$, $C_1, C_2, C_3, C_4$, such that for $p \in [1, \infty]$, for  every $\lambda \geq 1$, for every $\varphi_{\lambda} \in C^{\infty}(M)$ satisfying \eqref{eq:Laplaceeingenfunction}, for every $x \in M$, the following estimate holds
\begin{align}
 & \|\varphi_\lambda(y)\|_{L^p\left(r\left(1+C_3\frac{\log(\sqrt{\lambda})}{\sqrt{\lambda}}\right) < d_g(y,x) < r\left(1+C_4\frac{\log(\sqrt{\lambda})}{\sqrt{\lambda}}\right)\right)} \notag\\
  &\leq C \lambda^{\beta} \|\varphi_\lambda(y)\|_{L^p\left(r\left(1+C_1\frac{ \log(\sqrt{\lambda})}{\sqrt{\lambda}}\right) <  d_g(y,x) < r\left(1+C_2\frac{\log(\sqrt{\lambda})}{\sqrt{\lambda}}\right)\right)}, \ q_d =C(d, C_1, C_2, C_3, C_4)>0. \label{eq:LpBernsteinfunctionAnnulusharmonicSecondPraticalLaplace}
 \end{align}
\end{pr}
The proof of \Cref{tm:mainresult1} now mainly follows \Cref{sec:almostsharp} and \Cref{sec:endproofharmonic}. One of the main difference is that we are not still working with the $A$-harmonic equation $\mathrm{div}(A(x) \nabla u) = 0$ but instead with the equation $-\Delta_g u - \lambda u = 0$. But, because the elliptic regularity estimates are performed in small balls of radius $\varepsilon = \log(\sqrt{\lambda})/\sqrt{\lambda}$, then the lower order term part $\lambda u$ appearing in the estimates with only a logarithm loss does not affect the overall strategy.

\section{Extensions}
\label{sec:extensions}

The goal of this part is to give possible several generalizations of our main results \Cref{tm:mainresult1} and \Cref{tm:mainresult2}. For the sake of simplicity, we only give sketches of the proof, the details will be omitted. We also propose some open problems that can be investigated in the future.

\subsection{Manifolds with boundaries}
 The treatment of $C^{\infty}$-manifolds $M$, possibly with boundaries i.e. $\partial M \neq \emptyset$, can also be done. So analogues of \Cref{tm:mainresult1} hold, for Laplace eigenfunctions
\begin{equation*}
- \Delta_g \varphi_{\lambda} = \lambda \varphi_{\lambda}\ \text{in}\ M,\ (\varphi_{\lambda} = 0\ \text{on}\ \partial M)\ \text{or}\ (\partial_{\nu} \varphi_{\lambda} = 0\ \text{on}\ \partial M).
\end{equation*}
The proofs are obtained from \Cref{tm:mainresult1} and the double manifold trick, see for instance \cite{DF90b} or more precisely \cite[Section 3]{BM23}, that consists in reducing the
question to the case of a manifold without boundary by gluing two copies of $M$ along
the boundary in such a way that the new double manifold $\tilde{M}$ inherits a Lipschitz metric, which allows one to apply the previous results (without boundary) to this double manifold.

\subsection{On elliptic differential inequalities} First, we have the following generalization of local $L^2$-Berstein estimates.
\begin{tm}
\label{tm:mainresult1Generalization}
There exist $r_0, C >0$ depending only on $M$, such that for function $\varphi \in C^{\infty}(M)$, satisfying the elliptic differential inequality
\begin{equation}
\label{eq:functionGeneralization}
|- \Delta_g \varphi| \leq \lambda |\varphi| + \mu |\nabla \varphi|\ \text{in}\ M,
\end{equation}
where $\lambda, \mu \geq 1$, then for every $x \in M$, $r \in (0,r_0)$,
\begin{equation}
\label{eq:LinftyBernsteinGeneralization}
 \|\varphi\|_{L^2\left(B_g\left(x,r\left(1+\min\left(\frac{1}{\lambda^{2/3}}, \frac{1}{\mu^{2}} \right) \right)\right)\right)}\leq C  \|\varphi\|_{L^2(B_g(x,r))},
\end{equation}
and
\begin{equation}
\label{eq:LinftyBernsteingradientGeneralization}
\|\nabla \varphi\|_{L^2\left(B_g\left(x,r\left(1+\min\left(\frac{1}{\lambda^{2/3}}, \frac{1}{\mu^{2}} \right) \right)\right)\right)} \leq C \frac{\max\left(\lambda^{2/3}, \mu^2\right)}{r} \|\varphi\|_{L^2\left(B_g\left(x,r\right)\right)}.
\end{equation}
\end{tm}
The proof of \Cref{tm:mainresult1Generalization} cannot use the standard lifting trick from \eqref{eq:liftingtrick} so one needs to proceed differently. \medskip

First, by a standard $L^2$-Carleman estimate and the arguments of \cite{DF88}, the proof consists in establishing the following vanishing order estimate for $\varphi$.\medskip

\noindent \textbf{Bound on the doubling index.} There exist $r_0, C >0$ depending only on $M$, such that for every function $\varphi \in C^{\infty}(M)$ satisfying \eqref{eq:functionGeneralization}, for every $x \in M$, $r \in (0,r_0)$,
\begin{equation}
\label{eq:doublingindexLaplaceeigenGeneralizationBis}
\sup_{B_g\left(x,2 r\right)} |\varphi| \leq e^{C \max\left(\lambda^{2/3}, \mu^2\right)} \sup_{B_g\left(x,r\right)} |\varphi|.
\end{equation}

Secondly, by the application of a $L^2$-Carleman estimate on a punctured geodesic ball, taking $\alpha \geq C \max\left(\lambda^{2/3}, \mu^2\right)$ to absorb the right hand side terms and by using the arguments of the proof of \cite[Theorem 1]{DF90} together with \eqref{eq:doublingindexLaplaceeigenGeneralizationBis}, one is able to prove \eqref{eq:LinftyBernsteinGeneralization} then \eqref{eq:LinftyBernsteingradientGeneralization}. $L^p$-versions can also be obtained, with an arbitrary small loss.\medskip

Note that \eqref{eq:LinftyBernsteinGeneralization} and \eqref{eq:LinftyBernsteingradientGeneralization} are crucially related to the vanishing order estimate \eqref{eq:doublingindexLaplaceeigenGeneralizationBis} that is itself related to the possible rate of decay to the solutions of second order differential inequalities in the Euclidean space $\R^d$ inspired by the so-called Landis conjecture \cite{KL88}. According to the Meshkov type counterxamples to the Landis conjecture for complex-valued functions \cite{Mes91} and \cite{Dav14}, \eqref{eq:doublingindexLaplaceeigenGeneralizationBis} is probably sharp so are \eqref{eq:LinftyBernsteinGeneralization} and \eqref{eq:LinftyBernsteingradientGeneralization} if we consider $\varphi \in C^{\infty}(M;\C)$. But, one can probably sharpen \eqref{eq:LinftyBernsteinGeneralization} and \eqref{eq:LinftyBernsteingradientGeneralization} when $d=2$, assuming that $\varphi$ is real-valued by using the recent paper on vanishing order estimates of real-valued solutions to second order elliptic equations in the plane,  \cite{LMNN20} or \cite{LBS23}. When $d \geq 3$, and assuming that $\varphi$ is real-valued, we do not know if one can sharpen \eqref{eq:LinftyBernsteinGeneralization} and \eqref{eq:LinftyBernsteingradientGeneralization} but in that direction, one can read the very recent preprint \cite{FK24} that constructs a real-valued counterexample to the quantitative Landis conjecture in $\R^d$ for $d \geq 4$.\medskip

Possible generalizations can also be done for solutions, with bounded doubling index, of 
\begin{equation*}
\label{eq:AharmonicGene}
- \mathrm{div}(A(x) \nabla u) + W \cdot \nabla u + V u = 0\ \text{in}\ B_2,
\end{equation*}
where $A=(a^{ij}(x))_{1 \leq i, j \leq d}$ symmetric, uniformly elliptic, with Lipschitz entries, that is satisfying \eqref{eq:LipschitzAB2}. The lower order terms are given by $W  = W(x) \in L^{\infty}(B_2;\C^d)$ and $V \in L^{\infty}(B_2;\C)$ satisfying
\begin{equation*}
\label{eq:perturbationterm}
|W(x)| + |V(x)| \leq \Lambda_3,\qquad x \in B_2,
\end{equation*}
for some $\Lambda_3 >0$.

\subsection{Uniform bound on the doubling index on spheres}

We focus on $A$-harmonic functions $u$ in $B_2$, i.e.
\begin{equation}
\label{eq:AharmonicFreq}
- \mathrm{div}(A(x) \nabla u) = 0\ \text{in}\ B_2,
\end{equation} 
with bounded doubling index
\begin{equation}
\label{eq:bounddoublingindexFreq}
N_u(B(0,1)) := \log\left(\frac{\sup_{B\left(0,2\right)} |u|}{\sup_{B\left(0,1\right)} |u|}\right) \leq N,\qquad N \geq 1.
\end{equation}

In the spirit of \Cref{tm:df90harmonic}, we have the following result at $L^2$-level, concerning the doubling index on spheres.
\begin{tm}
\label{tm:df90harmonicSpheres}
There exist $r_0, C >0$ depending only on $A$ such that for every $A$-harmonic function $u \in H_{\text{loc}}^1(B_2) \cap L^{\infty}(B_2)$, i.e. satisfying \eqref{eq:AharmonicFreq} with a bounded doubling index $N$ defined in \eqref{eq:bounddoublingindexFreq}, for every $r \in (0,r_0)$,
\begin{align}
\label{eq:L2BernsteinfunctionAnnulusharmonicSpheres}
 \|u(x)\|_{L^2\left(|x|= r\left(1+\frac{1}{N}\right)\right)} \leq C  \|u(x)\|_{L^2\left(|x|=r\right)}.
\end{align}
\end{tm}
The proof is based on the frequency function approach. We follow \cite[Section 2]{LM20}.
\begin{proof}
Let $u$ be a solution to the equation $- \mathrm{div}(A(x)\nabla u(x))=0$. We consider weighted averages of $|u|^2$ over spheres
\[H(r)=r^{1-d}\int_{\partial B_r}\mu(x)|u(x)|^2ds(x).\]
We then define
\begin{equation}
\label{eq:INH}
I(r)=r^{1-d}\int_{B_r}(A\nabla u, \nabla u)=r^{-d}\int_{\partial B_r}(uA\nabla u,x),\quad N(r)=\frac{rI(r)}{H(r)}.
\end{equation}
Then 
\begin{equation}\label{eq:H'}
H'(r)=2I(r)+\mathcal{O}(C H(r)),\quad N(r)=\frac{rH'}{2H}+\mathcal{O}(C r).
\end{equation}
\noindent \textbf{Fact 1.} There exists $C>0$  that depends only on the dimension and $A$ such that the function $e^{C r}N(r)$ is an increasing function of $r$.\medskip

\noindent \textbf{Fact 2.} We have that $N(R) \leq C (N+1)$.\medskip

For any $r<R<R_0/2$, we use \eqref{eq:H'}, \eqref{eq:INH} and apply Fact 1
\[H'(r)\le 2I(r)+cH(r)=(2r^{-1}N(r)+c)H(r)\le (2r^{-1}N(R)e^{C(R-r)}+c)H(r).\]
Integrating $H'(r)/H(r)$ over an interval $[\rho, \rho(1+N(R)^{-1})]$ we get
\[\int_{\partial B_{\rho(1+N(R)^{-1})}}\mu(x)|u(x)|^2ds(x)\le C \int_{\partial B_{\rho}}\mu(x)|u(x)|^2ds(x),\]
The conclusion of the proof of \eqref{eq:L2BernsteinfunctionAnnulusharmonicSpheres} follows from Fact 2.
\end{proof}

We do not know if it is possible to obtain the generalization of \Cref{tm:df90harmonicSpheres} for $L^p$-norms. \medskip

The extension of \Cref{tm:df90harmonicSpheres} to Laplace eigenfunctions is probably possible by reducing the equation $-\Delta_g u - \lambda  u$ to $-\mathrm{div}(A(x) \nabla u) = 0$ in the geodesic annulus $r < |x| < r(1+\sqrt{\lambda}^{-1})$, that has Poincaré constant bounded by $C r \sqrt{\lambda}^{-1}$, so for $r \in (0,r_0)$ sufficiently small, one can construct a positive multiplier to the equation $-\Delta_g \phi - \lambda  \phi = 0$, then set $v = u/\phi$ that satisfies $-\mathrm{div}(\phi^2 A(x) \nabla v) = 0$. A careful inspection of the constant appearing in the proof of the frequency function approach depending on $A$ and $\phi$, then on $\lambda$ should be tracked.

\subsection{Linear combination of eigenfunctions} An interesting open problem is the generalization of local $L^{\infty}$-Bernstein estimates for linear combination of eigenfunctions. While global $L^{\infty}$-Bernstein estimates \eqref{eq:globalBernsteinSum} have been established, it seems that its local counterpart has not been investigated yet. Let us recall that from \cite[Theorem 14.3]{JL99}, the following result holds. \medskip

\noindent \textbf{Bound on the doubling index.} There exist $r_0, C >0$ depending only on $M$, such that for every linear combination of Laplace eigenfunctions
\begin{equation}
\label{eq:deflinearcombination}
\Phi_{\Lambda} = \sum_{\lambda_k \leq \Lambda} a_k \varphi_{\lambda_k},\qquad a_k \in \C,\ \Lambda \geq 1,\ \text{with}\ -\Delta_{g} \varphi_{\lambda_k} = \lambda_k \varphi_{\lambda_k}\ \text{in}\ M,
\end{equation}
then for every $x \in M$, $r \in (0,r_0)$,
\begin{equation}
\label{eq:doublingindexLaplaceeigenLinear}
\sup_{B_g\left(x,2 r\right)} |\Phi_{\Lambda}| \leq e^{C \sqrt{\Lambda}} \sup_{B_g\left(x,r\right)} |\Phi_{\Lambda}|.
\end{equation}

It is then natural to conjecture the following result.
\begin{conj}
\label{conj:sumeigenfunctions}
There exist $r_0, C >0$ depending only on $M$ such that for every linear combination of Laplace eigenfunctions $\Phi_{\Lambda}$, that is satisfying \eqref{eq:deflinearcombination}, the following Bernstein estimates hold
\begin{equation}
\label{eq:LinftyBernsteinGeneralizationLinear}
\sup_{B_g\left(x,r\left(1+\frac{1}{\sqrt{\Lambda}}  \right)\right)} |\Phi| \leq C \sup_{B_g\left(x,r\right)} |\Phi|,
\end{equation}
and
\begin{equation}
\label{eq:LinftyBernsteingradientGeneralizationLinear}
\sup_{B_g\left(x,r\right)} |\nabla \Phi| \leq C \frac{\sqrt{\Lambda}}{r} \sup_{B_g\left(x,r\right)} |\Phi_|.
\end{equation}
\end{conj}
The main difficulty for obtaining \Cref{conj:sumeigenfunctions} comes from the fact that $\Phi$ does not satisfy an elliptic equation. The standard trick to remove this difficulty consists in setting
\begin{equation}
\label{eq:tricklineareigen}
\tilde{\Phi}(x,t) = \sum_{\lambda_k \leq \Lambda} a_k \frac{\sinh(\sqrt{\lambda_k t})}{\sqrt{\lambda_k}} \varphi_{\lambda_k}\qquad (x,t) \in M \times \R,
\end{equation}
because $\Delta\tilde{\Phi}$ is harmonic in $M \times \R$, with respect to the metric $\tilde{g} = g \otimes \mathrm{dt}$. This transformation \eqref{eq:tricklineareigen} is crucially used for proving \eqref{eq:doublingindexLaplaceeigenLinear}. A first attempt for proving \eqref{eq:LinftyBernsteinGeneralizationLinear} would to consider a boundary $L^2$-Carleman-type inequality in the spirit of \cite[Lemma 14.5]{JL99}. Indeed, this boundary type estimate is useful for deducing an estimate of $\Phi$ from an estimate of $\tilde{\Phi}$. Note that even the $L^2$-version of \Cref{conj:sumeigenfunctions} is open.

\appendix

\section{Proof of the Carleman estimates}
\label{sec:proofCarleman}

In this part, we introduce standard notations inspired by Riemannian geometry. We do this because it simplifies the formulas appearing in the proof of the next lemmas. We set 
\begin{equation*}
\partial_i f = \frac{\partial f}{\partial x_i},\ \nabla_x f = (\partial_1 f, \dots, \partial_d f),\ \nabla f = A \nabla_x f,\
\mathrm{div}(\xi) = \sum_{i=1}^{d} \partial_i \xi_i\ \text{and}\ \Delta f = \mathrm{div}(\nabla f),
\end{equation*}

Let $A^{-1} = (a_{ij}(x))_{1 \leq i, j \leq d}$ the inverse matrix of coefficients of $A$.

For two vector fields $\xi$ and $\eta$, we have
\begin{equation}
\label{eq:scalarproductmetric}
\xi \cdot \eta = \sum_{i,j=1}^{d} a_{ij}(x) \xi_i \eta_j,\ |\xi|^2 = \xi \cdot \xi,
\end{equation}

With these notations at hand, when $f$, $h$ are smooth compactly supported functions, we have
\begin{equation*}
\mathrm{div}(A(x) \nabla f) = \Delta f,\ \Delta (f^2) = 2 f \Delta f + 2 |\nabla f|^2,\ \int f \Delta h dx = \int h \Delta f dx = - \int \nabla f \cdot \nabla h dx.
\end{equation*}

The proof of  \Cref{lm:Carleman} will follow the lines of the one of \cite[Theorem 2]{EV03}, by keeping in the left-hand-side the whole anti-symmetric term of the conjugated operator, see \Cref{lm:Carleman_app} for a precise formulation. Indeed, this term will then be exploited to give a more direct proof of \Cref{lm:Carlemanpunctured} than the one of \cite[Lemma A]{DF90} in the simple case $A=I_d$.

\subsection{The standard Carleman estimate}

The main result of this section is the following Carleman estimate, that leads to \Cref{lm:Carleman}.
\begin{lm}
\label{lm:Carleman_app}
There exists a positive constant $C=C(A)>0$, a radial increasing function $\rho=\rho(r)$ for $0 < r < 2$ satisfying
\begin{equation}
\label{eq:assumptionw}
C^{-1} \leq \frac{\rho(r)}{r} \leq C,\ C^{-1}  \leq |\partial_r \rho(r) | \leq C\ \qquad \forall r \in (0,2),
\end{equation}
such that for every $\alpha \geq C$, $f \in C_c^{\infty}(B_{2} \setminus \{0\})$, the following estimate holds
\begin{multline}
\label{eq:Carleman1_app}
\alpha^3 \int_{B_{2}} \rho^{-1-2 \alpha} |f|^2 dx + \alpha  \int_{B_{2}} \rho^{1- 2 \alpha} |\nabla_x f|^2 dx + \alpha^2 \int_{B_{2}} \frac{|\nabla \rho|^2}{\rho^2} |\mathcal A(g)|^2 dx \\
\leq C  \int_{B_{2}} \rho^{2-2 \alpha} |\Delta f|^2 dx,
\end{multline}
where $g = \rho^{-\alpha} f$ and
\begin{equation}
\label{eq:definitionAntisymmetric}
\mathcal A(g) = \frac{\rho \nabla \rho \cdot \nabla g}{|\nabla \rho|^2} + \frac{1}{2} F_{\rho}^A g,\qquad F_{\rho}^A = \frac{\rho \Delta \rho - |\nabla \rho|^2}{|\nabla \rho|^2},
\end{equation}
together with the bound
\begin{equation}
\label{eq:BoundFwA}
|F_{\rho}^A| \leq C.
\end{equation}
\end{lm}

\begin{proof}
We follow \cite[Theorem 2]{EV03}, sometimes line by line.

Let $g = \rho^{-\alpha} f$ and we compute
\begin{equation}
\label{eq:identityconjugaison}
\rho^{-\alpha} \Delta f = \Delta g + \frac{\alpha^2 |\nabla \rho|^2}{\rho^2} g+ 2 \alpha \frac{|\nabla \rho|^2}{\rho^2} \mathcal{A}(g),
\end{equation}
where the definition of $\mathcal{A}(g)$ is recalled in \eqref{eq:definitionAntisymmetric}. We also set $M_{\rho}^A$ the $d \times d$ symmetric matrix
\begin{equation}
\label{eq:def1Mw}
M_{\rho}^A = \frac{1}{2} (M_{ij} + M_{ji})_{1 \leq i,j \leq d},
\end{equation}
where, using the summation notation of repeated indices,
\begin{equation}
\label{eq:def2Mw}
M_{ij}  =\frac{1}{2} \mathrm{div} \left(\frac{\rho \nabla \rho}{|\nabla \rho|^2}\right) \delta_{ij} - \partial_{x_j} \left( \frac{\rho a^{ik} \partial_{x_k} \rho}{|\nabla \rho|^2}\right) + \frac{1}{2} a_{ik} \frac{\rho a^{kl} \partial_l \rho}{|\nabla \rho|^2} \partial_k a^{hi} - \frac{1}{2} F_{\rho}^A \delta_{ij}.
\end{equation}

We split the proof in several steps.\medskip

\textbf{Step 1: A first identity.} The goal of this step is to prove that
\begin{equation}
\label{eq:goalstep1}
\int \frac{\rho^2}{|\nabla \rho|^2} (\rho^{- \alpha} \Delta f)^2 \geq 4 \alpha \int M_{\rho}^A  \nabla g \cdot \nabla g + \alpha \int F_{\rho}^A \Delta (g^2) + 4 \alpha^2 \int \frac{|\nabla \rho|^2}{\rho^2} \mathcal A(g)^2.
\end{equation}

Let
\begin{equation}
\label{eq:Step11}
J_1 = \int \left(2 \alpha \frac{|\nabla \rho|}{\rho} \mathcal A(g)\right)^2,\ J_2 = 2 \int \left[ 2 \alpha \mathcal A(g) \left( \Delta g + \alpha^2 \frac{|\nabla \rho|^2}{\rho^2} g  \right) \right].
\end{equation}
Then, we have
\begin{equation}
\label{eq:Step12}
\int \frac{\rho^2}{|\nabla \rho|^2} (\rho^{-\alpha} \Delta f)^2 \geq J_1 + J_2.
\end{equation}
First note that
\begin{equation}
\label{eq:Step13}
\int \frac{ |\nabla \rho|^2}{\rho^2} \mathcal A(g) \cdot g = 0,
\end{equation}
so that from \eqref{eq:definitionAntisymmetric} and the identity $\Delta (g^2) = 2 g \Delta g + 2 |\nabla g|^2$,
\begin{equation}
\label{eq:J2int}
J_2 = 4 \alpha \int \mathcal A(g) \Delta g  = 2 \alpha \int \left( \frac{2 \rho \nabla \rho \cdot \nabla g}{|\nabla \rho|^2} \Delta g - F_{\rho}^A  |\nabla g|^2 \right) + \alpha \int F_{\rho} \Delta (g^2).
\end{equation}
We now have the following identity
\begin{multline*}
2 ( \beta \cdot \nabla g) \Delta g = 2 \mathrm{div} ((\beta \cdot \nabla g) \nabla g) - \mathrm{div}(\beta |\nabla g|^2) + \mathrm{div}(\beta) |\nabla g|^2 \\
- 2 \partial_i \beta^k a^{ij} \partial_j g \partial_k g + \beta^k \partial_k a^{ij} \partial_i g \partial_j g,\qquad \forall \beta \in C^{\infty}(B_2 \setminus \{0\} ; \R^d),
\end{multline*}
and we choose
\begin{equation*}
\beta = \frac{\rho \nabla \rho}{|\nabla \rho|^2},
\end{equation*}
to get from \eqref{eq:J2int} and the divergence theorem
\begin{equation}
\label{eq:Step14}
4 \alpha \int \mathcal A(g) \Delta g = 4 \alpha \int M_{\rho}^A  \nabla g \cdot \nabla g + \alpha \int F_{\rho}^A  \Delta (g^2),
\end{equation}
where $M_{\rho}^A $ is defined in \eqref{eq:def1Mw} and \eqref{eq:def2Mw}. By gathering
\eqref{eq:Step11}, \eqref{eq:Step12}, \eqref{eq:Step13} and \eqref{eq:Step14} we obtain \eqref{eq:goalstep1} so the conclusion of Step 1.\\
\medskip

\textbf{Step 2: Choice of $\rho$.} For $\mu \geq 1$, a parameter to be chosen later, let
\begin{equation}
\label{eq:defsigmavarphiphi}
\sigma(x) = \left(\sum_{i,j=1}^{d} a_{ij}(0) x_i x_j\right)^{1/2},\ \varphi(s) = s \exp\left( \int_0^s \frac{e^{-\mu t}-1}{t} dt\right),\ \phi(s) = \frac{\varphi(s)}{s \varphi'(s)} = e^{\mu s}.
\end{equation}
We define
\begin{equation*}
\rho(x) = \varphi(\sigma(x)).
\end{equation*}
With this definition, we can now compute
\begin{equation}
\label{eq:Msigmanablasigma}
M_{\sigma}^A \nabla \sigma = 0,
\end{equation}
\begin{equation}
\label{eq:FwMwChoice}
F_{\rho}^A = F_{\sigma}^A \phi(\sigma) - \sigma \phi'(\sigma),\ M_{\rho}^A = \phi(\sigma) M_{\sigma}^A +  \sigma \phi'(\sigma) \left( I- \frac{\nabla \sigma \otimes \nabla \sigma}{|\nabla \sigma|^2} \right),
\end{equation}
and
\begin{equation}
\label{eq:FMsigma0}
F_{\sigma}^{A(0)} = n-2,\ M_{\sigma}^{A(0)} = 0.
\end{equation}
Notice that the following properties hold
\begin{equation*}
\varphi'(r) > 0,\ cr \leq \varphi(r) \leq Cr,
\end{equation*}
so we have
\begin{equation*}
c \sigma \leq \rho(x) \leq C \sigma,\ c  \leq |\nabla \rho(x) | \leq C,\ |\nabla |\nabla \rho|^2| \leq C,\ |\Delta \phi | \leq C \rho^{-1},\ |F_{\rho}| \leq C.
\end{equation*}

Now we estimate the first two terms appearing in the right hand side of \eqref{eq:goalstep1}.

Let us treat the first term. From the second part of \eqref{eq:FwMwChoice}, we have
\begin{equation*}
M_{\rho}^A \nabla g \cdot \nabla g = \sigma \phi' \left(|\nabla g|^2 - \frac{(\nabla \sigma \cdot \nabla g)^2}{|\nabla \sigma|^2} \right) ,
\end{equation*}
and denoting by $\tilde{\nabla} g$ the tangential components of the gradient of $g$ along the level sets of $\sigma(x)$ with respect to the metric $\sum_{i,j=1}^{d} a_{ij}(x) dx_i dx_j$, we have
\begin{equation}
\label{eq:defnablatilde}
\tilde{\nabla} g = \nabla g - \frac{\nabla \sigma \cdot \nabla g}{|\nabla \sigma|^2} \nabla \sigma =  \nabla g - \frac{\nabla \rho \cdot \nabla g}{|\nabla \rho|^2} \nabla \rho.
\end{equation}
From \eqref{eq:Msigmanablasigma},  \eqref{eq:FMsigma0} and \eqref{eq:defnablatilde}, we have that
\begin{equation*}
M_{\sigma}^{A} \nabla g \cdot \nabla g = (M_{\sigma}^A - M_{\sigma}^{A(0)}) \tilde{\nabla} g \cdot \tilde{\nabla} g .
\end{equation*}
On the other hand, a computation and the Lipschitz condition on the matrix $A$ give that there exists $C>0$ depending on on $d$ and $\Lambda_1, \Lambda_2$ such that
\begin{equation*}
| M_{\sigma}^{A} - M_{\sigma}^{A(0)}| \leq C \sigma,
\end{equation*}
so that
\begin{equation}
\label{eq:FirsttermStep2}
\int M_{\sigma}^{A} \nabla g \cdot \nabla g \geq \int \sigma (\phi' - C \phi) |\tilde{\nabla} g|^2
\end{equation}

Now we estimate from below the second term appearing in \eqref{eq:goalstep1}. We observe that
\begin{equation*}
F_{\rho}^{A} = (d-2) \phi + (B \phi - \sigma \phi'),
\end{equation*}
where $B = F_{\sigma}^A - F_{\sigma}^{A(0)}$ that satisfies by using the Lipschitz condition on the matrix $A$,
\begin{equation*}
|B(x)| \leq C \sigma.
\end{equation*}
So we have from the identities
\begin{equation*}
\Delta g^2 = 2 g \Delta g + 2 |\nabla g|^2,\ |\nabla g|^2 = |\tilde{\nabla} g|^2 + \frac{(\nabla \rho \cdot \nabla g)^2}{|\nabla \rho|^2},
\end{equation*}
that
\begin{align}
\label{eq:SecondtermStep2}
\int F_{\rho}^A \Delta (g^2) &= (d-2) \int (\Delta \phi) g^2  + 2 \int (B \phi - \sigma \phi') g \Delta g \notag \\
&\quad + 2 \int (B \phi - \sigma \phi') |\tilde{\nabla} g|^2 + 2 \int  (B \phi - \sigma \phi') \frac{(\nabla \sigma \cdot \nabla g)^2}{|\nabla \sigma|^2}.
\end{align}
From \eqref{eq:identityconjugaison}, we then have that the second term of \eqref{eq:SecondtermStep2} writes as follows
\begin{multline}
\label{eq:SecondtermSecondStep2}
2 \int (B \phi - \sigma \phi') g \Delta g =  2 \int (B \phi - \sigma \phi') g \Delta f \rho^{-\alpha} \\
+ 2 \alpha^2 \int (\sigma \phi' - B \phi) \frac{|\nabla \rho|^2}{\rho^2} g^2 + 2 \int (\sigma \phi' - B \phi)  2 \alpha \frac{|\nabla \rho|^2}{\rho^2}  g \mathcal A(g) .
\end{multline}
Thus, from \eqref{eq:FirsttermStep2}, \eqref{eq:SecondtermStep2} and \eqref{eq:SecondtermSecondStep2} we have for $\alpha \geq 1$,
\begin{multline}
\label{eq:EstimateStep2}
4 \alpha \int M_{\rho}^A \nabla g \cdot \nabla g + \alpha \int F_{\rho} \Delta (g^2) \\
\geq 2 \alpha \int (\sigma \phi' - 2 C \sigma \phi + B \phi) |\tilde{\nabla} g|^2 + 2
\alpha^3 \int( \sigma \phi' - B \phi) \frac{|\nabla \rho|^2}{\rho^2} g^2 - R_1,
\end{multline}
where for some constant $C>0$ depending on $d$, $\Lambda_1, \Lambda_2$ and $\mu$,
\begin{equation}
\label{eq:EstimateStep2R1}
R_1 \leq C \left( \alpha \int \rho^{-1} g^2 + \alpha \int \rho^{1- \alpha} |g| |\Delta g| + \alpha^2 \int \rho^{-1} |\mathcal A(g)| |g| + \alpha \int \rho \frac{ |\nabla \sigma \cdot \nabla g|^2}{|\nabla \sigma|^2} \right).
\end{equation}
By choosing $\mu = C(d,\Lambda_1,\Lambda_2) \geq 1$ sufficiently large, we then get that for some $c=c(d,\Lambda_1,\Lambda_2,\mu)>0$  by using \eqref{eq:EstimateStep2},
\begin{equation}
\label{eq:EstimateStep2Bis}
4 \alpha \int M_{\rho}^A \nabla g \cdot \nabla g + \alpha \int F_{\rho} \Delta (g^2) \\
\geq c \left( \alpha \int \sigma |\tilde{\nabla} g|^2 + 2
\alpha^3 \int \sigma \frac{|\nabla \rho|^2}{\rho^2} g^2 \right) - R_1
\end{equation}
Once we combine \eqref{eq:goalstep1}, \eqref{eq:EstimateStep2Bis} and \eqref{eq:EstimateStep2R1}, we obtain the following
\begin{multline}
\label{eq:endStep2}
4 \alpha^2 \int \frac{|\nabla \rho|^2}{\rho^2} A(g)^2 + 2 \alpha \int \sigma \phi' |\tilde{\nabla} g|^2 + 2 \alpha^3 \int \sigma \phi' \frac{|\nabla \rho|^2}{\rho^2} g^2 \leq \int \frac{\rho^2}{|\nabla \rho|^2} ( \rho^{-\alpha} \Delta f)^2  \\
+ C \left( \alpha \int \rho^{-1} g^2 + \alpha \int \rho^{1- \alpha} |g| |\Delta g| + \alpha^2 \int \rho^{-1} |\mathcal A(g)| |g| + \alpha \int \rho \frac{ |\nabla \sigma \cdot \nabla g|^2}{|\nabla \sigma|^2} \right).
\end{multline}
This ends Step 2.\\
\medskip

\textbf{Step 3: Absorption.} To conclude the proof, recall the form of $\mathcal A(g)$ and then
\begin{equation}
\label{eq:BoundFwAProof}
|F_{\rho}^A| \leq C.
\end{equation}
Thus
\begin{equation}
\frac{\rho}{|\nabla \rho|^2} \frac{|\nabla \rho \cdot \nabla g|^2}{|\nabla \rho|^2} \leq C \rho^{-1} |\mathcal A (g) |^2 + C \frac{|g|^2}{\rho},
\end{equation}
so that
\begin{multline}
\label{eq:absorption1}
\alpha \int \rho \frac{ |\nabla \sigma \cdot \nabla g|^2}{|\nabla \sigma|^2} = \alpha \int \rho \frac{ |\nabla \rho \cdot \nabla g|^2} {|\nabla \sigma|^2} \leq C \alpha \int \rho^{-1} |\mathcal A (g)|^2 + C \alpha \int \rho^{-1} |g|^2 \\
\leq C \alpha \int \frac{|\nabla \rho|^2}{\rho} |\mathcal A  (g)|^2 + C \alpha \int \rho^{-1} |g|^2.
\end{multline}
Also,
\begin{equation}
\label{eq:absorption2}
C \alpha^2 \int \rho^{-1} |\mathcal A (g)| |g| \leq \frac{c}{2} \alpha^2 \int \frac{|\nabla \rho|^2}{\rho} |\mathcal A(g)|^2 + C \alpha^2 \int \rho^{-1} g^2,
\end{equation}
and
\begin{equation}
\label{eq:absorption3}
\alpha \int \rho^{(1- \alpha)} |g| |\Delta f| \leq \int \rho^{2} (\rho ^{-\alpha} |\Delta f|)^2 + C \alpha^2 \int  \rho^{-1} g^2.
\end{equation}
From \eqref{eq:endStep2}, \eqref{eq:absorption1}, \eqref{eq:absorption2} and \eqref{eq:absorption3}, conjugated with
$$  \alpha^2 \int \frac{|\nabla \rho|^2}{\rho} |A(g)|^2 \geq c \alpha^2 \int \frac{|\nabla \rho|^2}{\rho} |\mathcal A(g)|^2,$$
we obtain the expected inequality \eqref{eq:Carleman1_app} and the bound \eqref{eq:BoundFwA} comes from \eqref{eq:BoundFwAProof}. This concludes the proof.
\end{proof}


\subsection{The $L^2$-Carleman estimate in a punctured domain for the flat Laplacian}

The goal of this part consists in a simple proof of \Cref{lm:Carlemanpunctured} when $A=I_d$. As said previously, the proof of \Cref{lm:Carlemanpunctured} in the general case can be obtained by working with geodesic polar coordinates adapted to the metric given by $A$, see \cite{AKS62} and \cite{Rul18}, then following \cite{DF90}. The strategy that we employed here is a little bit different from the one of \cite[Lemma A]{DF90}. We actually exploit the square of the whole anti-symmetric term of the conjugated operator in the left hand side of the standard Carleman estimate.

\begin{proof}
We start from \eqref{eq:Carleman1} to obtain first
\begin{equation*}
 \alpha^2 \int_{B(0,r(1+2^{-1}\varepsilon))} \frac{|\nabla \rho|^2}{\rho^2} |\mathcal A(g)|^2 dx
\leq C  \int_{B_{2}} \rho^{2-2 \alpha} |\Delta f|^2 dx.
\end{equation*}
By using \eqref{eq:assumptionw} and the assumption on $f$ that vanishes in $B(0,r)$, this translates into
\begin{equation*}
 \alpha^2 r^{-2} \int_{B(0,r(1+2^{-1}\varepsilon))} |\mathcal A(g)|^2 dx
\leq C  \int_{B_{2}} \rho^{2-2 \alpha} |\Delta f|^2 dx.
\end{equation*}

Then we develop $|\mathcal A(g)|^2$ by \eqref{eq:definitionAntisymmetric} to get that
\begin{equation*}
 \alpha^2 r^{-2} \int_{B(0,r(1+2^{-1}\varepsilon))} \left(\frac{\rho \nabla \rho \cdot \nabla g}{|\nabla \rho|^2}\right)^2  \leq C  \int_{B_{2}} \rho^{2-2 \alpha} |\Delta f|^2 dx + C \alpha^2  r^{-2} \int_{B(0,r(1+2^{-1}\varepsilon))} |F_{\rho}^A g|^2.
\end{equation*}
By using that $A=I_d$, the definition of $\sigma$ in \eqref{eq:defsigmavarphiphi}, the definition of the scalar product between two vector fields in \eqref{eq:scalarproductmetric} we have that
\begin{equation*}
\nabla \rho \cdot \nabla g =  \partial_{\rho}  g.
\end{equation*}
We then deduce from the two previous estimates that
\begin{equation*}
\alpha^2 \int_{B(0,r(1+2^{-1}\varepsilon))} |\partial_{\rho} g|^2 
\leq C  \int_{B_{2}} \rho^{2-2 \alpha} |\Delta f|^2 dx + C \alpha^2  r^{-2} \int_{B(0,r(1+2^{-1}\varepsilon))} |F_{\rho}^A g|^2 
\end{equation*}
By integrating along a radial line and by using that $g$ vanishes in $B(0,r)$, we obtain that
\begin{equation*}
\alpha^2 \varepsilon^{-2} r^{-2} \int_{B(0,r(1+2\varepsilon))} | g|^2 \leq C  \alpha^2   \int_{B(0,r(1+2^{-1}\varepsilon))} |\partial_\rho g|^2,
\end{equation*}
so from the two previous estimates
\begin{equation*}
\alpha^2 \varepsilon^{-2}  r^{-2} \int_{B(0,r(1+2\varepsilon))} | g|^2 
 \leq C  \int_{B_{2}} \rho^{2-2 \alpha} |\Delta f|^2 dx + C \alpha^2  r^{-2} \int_{B(0,r(1+2^{-1}\varepsilon))} |F_{\rho}^A g|^2
\end{equation*}
Then, for $\alpha \geq C$, using \eqref{eq:BoundFwA} one can absorb the second right hand side term to obtain
\begin{equation*}
\alpha^2 \varepsilon^{-2} r^{-2} \int_{B(0,r(1+2\varepsilon))} | g|^2 \leq C  \int_{B_{2}} \rho^{2-2 \alpha} |\Delta f|^2 dx.
\end{equation*}
We now come back to the variable $f$ to deduce
\begin{equation}
\label{eq:CarlemanProof7}
\alpha^2 \varepsilon^{-2} r^{-2} \int_{B(0,r(1+2\varepsilon))} |f|^2 \rho^{-2 \alpha}
  \leq C  \int_{B_{2}} \rho^{2-2 \alpha} |\Delta f|^2 dx  .
\end{equation}
By gathering \eqref{eq:CarlemanProof7} with a Cacciopoli type estimate we obtain  the expected result \eqref{eq:secondCarleman}.
\end{proof}

\subsection{The modified version of the standard Carleman estimate}

We have the following Carleman estimate whose proof is an easy adaptation of the one of \Cref{lm:Carleman_app}.
\begin{lm}
\label{lm:Carleman_appModif}
There exists a positive constant $C=C(A)>0$ such that for every $x_0 \in B(0,1/4)$, there exists an increasing function $\rho$ such that
\begin{equation*}
\label{eq:assumptionwModif}
C^{-1} \leq \frac{\rho(|x-x_0|)}{|x-x_0|} \leq C,
\end{equation*}
such that for every $\alpha \geq C$, $f \in C_c^{\infty}(B_{2} \setminus \{x_0\})$, the following estimate holds
\begin{equation*}
\label{eq:Carleman1_appModif}
\alpha^3 \int_{B_{2}} \rho^{-1-2 \alpha} |f|^2 dx + \alpha  \int_{B_{2}} \rho^{1- 2 \alpha} |\nabla_x f|^2 dx
\leq C  \int_{B_{2}} \rho^{2-2 \alpha} |\Delta f|^2 dx.
\end{equation*}
\end{lm}
\begin{proof}
The proof exactly follows the lines of \Cref{lm:Carleman_app}. Only the choice of $\rho$, coming from \eqref{eq:defsigmavarphiphi}, is different. We instead set
$$ \sigma(x) = \left(\sum_{i,j=1}^{d} a_{ij}(x_0) (x_i-x_{0,i}) (x_j-x_{0,j})\right)^{1/2}.$$
Then the next estimates are established in function of $\rho$ that behaves as $|x-x_0|$.
\end{proof}

\bibliographystyle{alpha}
\small{\bibliography{Bernstein}}

\end{document}